\documentclass{imanum_arxiv}
\usepackage{graphicx}

\jno{drnxxx}
\received{\today}

\usepackage{mathtools}  
\usepackage{xfrac}
\usepackage{graphicx,graphics,epstopdf}
\usepackage{bm,bbm}
\usepackage{color}
\usepackage{pgfplots}

\newcommand{\usDinit}{{\ensuremath{u_{1}^0}}} 

\newcommand{\usAinit}{{\ensuremath{u_{2}^0}}}
\newcommand{\psinit}{{\ensuremath{p^0}}}
\newcommand{\dpsinit}{{\ensuremath{{\dot p}^0}}}

\newcommand{\cH}{\ensuremath{\mathcal{H}}}
\newcommand{\V}{\ensuremath{\mathcal{V}}}
\newcommand{\Q}{\ensuremath{\mathcal{W}} }
\newcommand{\Hbb}{\ensuremath{\mathbb{H}}}
\newcommand{\Vbb}{\ensuremath{\mathbb{V}}}
\def\N{\mathbb{N}}

\def\R{\mathbb{R}}

\definecolor{myBlue1}{RGB}{101,149,239}  
\definecolor{myBlue2}{RGB}{113,104,238} 
\definecolor{myBlue3}{RGB}{30,144,255} 
\definecolor{myGreen1}{RGB}{154,204,50} 
\definecolor{myGreen2}{RGB}{69,169,0} 
\definecolor{myGreen3}{RGB}{154,205,50} 
\definecolor{myGreen4}{RGB}{105,139,34} 
\definecolor{myRed1}{RGB}{210,105,30} 
\definecolor{myRed2}{RGB}{165,42,42} 
\definecolor{myRed3}{RGB}{139,26,26} 
\definecolor{lightgray}{RGB}{175,175,175} 
\definecolor{myLGray}{RGB}{225,225,225} 
\definecolor{mycolor0}{rgb}{0.66,0.66,0.66}
\definecolor{mycolor4}{rgb}{0.00000,0.44700,0.74100}
\definecolor{mycolor1}{rgb}{0.85000,0.32500,0.09800}
\definecolor{mycolor2}{rgb}{0.92900,0.69400,0.12500}
\definecolor{mycolor3}{rgb}{0.67000,0.74700,0.14100}
\definecolor{mycolor5}{rgb}{0.49400,0.18400,0.55600}
\definecolor{mycolor6}{rgb}{0.85000,0.32500,0.09800}%
\newcommand{\calB}{\ensuremath{\mathcal{B}} }

\newcommand{\calK}{\ensuremath{\mathcal{K}} }

\newcommand{\calN}{\ensuremath{\mathcal{N}} }

\def\dx{\,\text{d}x}

\newcommand{\partialkn}{\ensuremath{\partial_{\kappa,\nu}} }

\newcommand\calT{{\mathcal T}}

\renewcommand{\d}{\text{d}}

\newcommand{\Ga}{\Gamma}

\newcommand{\Om}{\Omega}
\newcommand{\pa}{\partial}

\def \to {\rightarrow}

\newcommand{\dofOm}{{\ensuremath{N_\Omega}}}
\newcommand{\dofGa}{{\ensuremath{N_\Gamma}}}

\newcommand{\cAM}{\alpha_\Omega}

\newcommand{\err}[1]{e_{#1}}

\renewcommand{\nu}{\textrm{n}}

\newtheorem{assumption}{\normalfont\textsc{Assumption}}[section]

\newcommand{\auskommentiert}[1]{ }

\allowdisplaybreaks
\begin{document}

\title{Bulk--surface Lie splitting for parabolic problems\\ with dynamic boundary conditions}
\shorttitle{Bulk--surface Lie splitting for parabolic problems with dynamic b.c.}

\author{%
{\sc Robert Altmann\thanks{Corresponding author. Email: robert.altmann@math.uni-augsburg.de}} \\[2pt]
Institute of Mathematics, University of Augsburg, \\
Universit\"atsstra{\ss}e 14, 86159 Augsburg, Germany \\[6pt]
{\sc and}\\[6pt]
{\sc Bal\'azs Kov\'acs}\thanks{Email: balazs.kovacs@mathematik.uni-regensburg.de}\\[2pt]
Faculty of Mathematics, University of Regensburg, \\
Universit\"atsstra{\ss}e 31, 93040 Regensburg, Germany \\[6pt]
{\sc and}\\[6pt]
{\sc Christoph Zimmer}\thanks{Email: christoph.zimmer@math.uni-augsburg.de}\\[2pt]
Institute of Mathematics, University of Augsburg, \\
Universit\"atsstra{\ss}e 14, 86159 Augsburg, Germany 
}
\shortauthorlist{R.~Altmann, B.~Kov\'acs, and C.~Zimmer}

\maketitle

\begin{abstract}
{This paper studies bulk--surface splitting methods of first order for (semi-linear) parabolic partial differential equations with dynamic boundary conditions. The proposed Lie splitting scheme is based on a reformulation of the problem as a coupled partial differential--algebraic equation system, i.e., the boundary conditions are considered as a second dynamic equation which is coupled to the bulk problem. The splitting approach is combined with bulk--surface finite elements and an implicit Euler discretization of the two subsystems. We prove first-order convergence of the resulting fully discrete scheme in the presence of a weak CFL condition of the form $\tau \leq c h$ for some constant~$c>0$. The convergence is also illustrated numerically using dynamic boundary conditions of Allen--Cahn-type.}
{dynamic boundary conditions, PDAE, splitting methods, bulk-surface splitting, parabolic equations.}
\end{abstract}


\section{Introduction}
This paper is devoted to the construction and analysis of a splitting method for parabolic initial-boundary value problems with dynamic boundary conditions. The aim of such a splitting method is to decouple bulk and surface dynamics, which is of particular interest for nonlinear or highly oscillatory boundary conditions. 
Analytical and modelling aspects of parabolic problems with dynamic boundary conditions have been studied extensively in recent years; see, e.g., \cite{CavGGM10,coclite2009stability,ColF15,engel2005analyticity,Favini2002heat,Gal2008well,Gal2008non,Goldstein2006derivation,goldstein2011cahn,Kenzler2001,Liero,Racke2003cahn,vazquez2011heat}. 

The numerical analysis of parabolic problems with dynamic boundary conditions has started with the work of~\cite{Fairweather} who proved the first error estimates for (conforming) Galerkin methods for the linear case. For a long time, however, his work went unnoticed in the dynamic boundary conditions community, possibly due to the fact that the term \emph{dynamic} has not appeared at all in his paper.
We refer to~\cite{KovL17} for the numerical analysis of general parabolic problems with dynamic boundary conditions using bulk--surface finite elements, including surface differential operators, semi-linear problems, and time integration. 
For the analysis and numerical analysis of the corresponding wave-type systems, i.e., wave equations with dynamic boundary conditions see~\cite{Hip17}, and the references therein for theoretical, modelling, and numerical aspects. Numerical methods for corresponding semi-linear systems were considered in~\cite{HochbruckLeibold2020,HochbruckLeibold2021}.

\medskip
The main motivation for this paper are the bulk--surface splitting experiments of~\cite[Sect.~6.3]{KovL17}. Therein, two Strang splitting schemes were proposed. Numerical experiments, however, show that both methods suffer from order reduction and give sub-optimal convergence rates, if any, see Figure~\ref{fig:forceComponent}. Corresponding experiments in Section~\ref{sec:numerics:Lie:naive} of this paper further show that the splitting methods are inadequate in the sense that they seem to approximate a perturbed system rather than the exact solution. 

In this paper, we propose a novel bulk--surface splitting method of first order.
The scheme is derived based on a reformulation of the problem as a coupled partial differential equation system, a so-called \emph{partial differential--algebraic equation} (PDAE); see~\cite{Alt19}. In the reformulated system, there are two evolution equations for dynamic variables (in the bulk and on the boundary) which are coupled with the help of the trace operator and a Lagrange multiplier. 
The coupled system is discretized in space using the bulk--surface finite element method; see~\cite{ElliottRanner,KovL17}.
The scheme is then derived using the matrix--vector formulation of the spatial semi-discretization, splitting the resulting differential--algebraic equation (DAE), and eliminating the Lagrange multiplier again. 
In contrast to the splitting methods introduced in~\cite{KovL17}, the proposed Lie splitting scheme also involves the (discrete) time derivatives of the solution. 
The sub-flows are fully discretized using the backward Euler method. The resulting fully discrete scheme is proved to be convergent of order one. 
For the convergence proof, we separately prove stability and consistency with the main issue being stability. The key idea of the proof is to rewrite the proposed method as a perturbation of the implicit Euler method applied to the semi-discrete DAE system and consider corresponding energy estimates, testing with the discrete time derivative of the error.

The resulting bulk--surface splitting is of particular interest if the solution oscillates rapidly on the boundary (see~\cite{AltV21} for an example) or if the boundary conditions are nonlinear. In the latter case, the approach allows, e.g., to solve a \emph{linear} system in the bulk and a nonlinear system only on the surface. Hence, the dimension of the nonlinear system is reduced significantly. 

The paper is outlined as follows. 
In Section~\ref{sec:formulation} we introduce the problem of interest and discuss different (weak) formulations. Moreover, we recall the existing bulk--surface splitting methods from~\cite{KovL17}. 
The spatial discretization of the system is then analyzed in Section~\ref{sec:formulation:space}. 
It follows the construction of the novel bulk--surface Lie splitting approach in Section~\ref{sec:Lie}. Here, we discuss and motivate the method and prove convergence under the assumption of a weak CFL condition. 
In Section~\ref{sec:numerics} we validate the theoretical findings by a number of numerical experiments, including dynamic boundary conditions of Allen--Cahn-type. 
Finally, we give some conclusions in Section~\ref{sec:conclusion}.
%
%
\section{Abstract formulations and naive splitting methods}
\label{sec:formulation}
As model problem, we consider parabolic problems in a bounded bulk domain $\Om \subseteq \R^d$ with dynamic boundary conditions on the boundary~$\Gamma\coloneqq\partial\Omega$. More precisely, we seek for $u\colon\overline{\Omega} \to \R$ such that 
\begin{subequations}
	\label{eq:dynamicBC}
	\begin{align}
	\dot  u - \nabla\cdot(\kappa\, \nabla u) + \alpha_\Omega u &= f_\Omega(u) \qquad\text{in } \Omega, \label{eq:dynamicBC:a} \\
	\dot u - \beta\, \Delta_\Gamma u + \partialkn u + \alpha_\Gamma u &= f_\Gamma(u) \qquad \text{on } \Gamma \label{eq:dynamicBC:b}
	\end{align}
\end{subequations}
with initial condition~$u(0) = u^0$. 
By~$\nu$ we denote the unit normal vector and the corresponding normal derivative reads~$\partialkn u \coloneqq \nu\cdot(\kappa\, \nabla u)$. The Laplace--Beltrami operator is denoted by~$\Delta_\Gamma$; cf.~\cite[Ch.~16.1]{GilT01}. 
For the parameters, we assume~$\kappa\in L^\infty(\Omega)$ with $\kappa(x)\ge c_\kappa > 0$ and constants~$\alpha_\Omega$, $\alpha_\Gamma \in \R$ as well as~$\beta\ge 0$. Without loss of generality, we further assume $\alpha_\Omega, \alpha_\Gamma \geq 0$ and hide the terms in the respective nonlinearities otherwise. The boundary condition~\eqref{eq:dynamicBC:b} is called \emph{Wentzell boundary condition} or \emph{locally reacting} for $\beta = 0$ and \emph{non-local} otherwise. 

The nonlinearities on the right-hand sides, which are assumed to be sufficiently regular, make the problem semi-linear. We are particularly interested in cases where the bulk and boundary evolution are substantially different. 

Before we engage in splitting approaches, we need to discuss abstract formulations of the bulk--surface problem. Here, we consider two different approaches: First, we formulate~\eqref{eq:dynamicBC} as an abstract parabolic problem. Second, we derive a formulation of~\eqref{eq:dynamicBC} as a coupled system. 
%
%
\subsection{Formulation as parabolic problem}
A general abstract framework for the weak formulation of semi-linear parabolic problems with dynamic boundary conditions was presented in~\cite[Sect.~2]{KovL17}. 
The corresponding Gelfand triple $\Vbb \subseteq \Hbb \subseteq \Vbb'$ is given by
\begin{align*}
\Vbb \coloneqq \{ v \in H^1(\Omega)\ |\ \sqrt{\beta}\, \gamma v \in H^1(\Gamma) \} , \qquad 
\Hbb \coloneqq L^2(\Om) \times L^2(\Ga) 
\end{align*}
with the continuous embedding $u \mapsto (u,\gamma u)$, where $\gamma$ is the usual trace operator. Sometimes, we will abbreviate pairs $(u,\gamma u) \in \Hbb$ by simply writing $u \in \Hbb$.

Next, we define bilinear forms $a\colon\Vbb \times \Vbb\to \R$ and $m\colon \Hbb \times \Hbb\to \R$. Without writing the trace operator explicitly, they are given by 
\begin{align*}
a(u,v) 
&\coloneqq \int_\Omega \alpha_\Omega  u\, v +  \kappa\, \nabla u \cdot \nabla v  \dx + \int_\Gamma \alpha_\Gamma u\, v +  \beta\, \nabla_\Gamma u \cdot \nabla_\Gamma v \d \sigma, \\ 
m((u,u_\Ga),(v,v_\Ga)) 
&\coloneqq \int_\Omega u\, v \dx + \int_\Gamma u_\Gamma \, v_\Gamma  \d \sigma.
\end{align*}
We will also use the bilinear forms separating the bulk and surface integrals, i.e., 
\begin{equation}
\label{eq:bilinear form splitting}
a = a_\Omega + a_\Gamma, \qquad m = m_\Omega + m_\Gamma .
\end{equation}

The weak formulation of~\eqref{eq:dynamicBC} then reads: Find $u \in C([0,T],\Hbb) \cap L^2([0,T],\Vbb)$ such that 
\begin{align*}
m(\dot  u, v) + a(u, v) 
= m(f(u),v)
\end{align*}
holds for any $v \in \Vbb$ and $t \in (0,T]$ and with initial condition~$u(0) = u^0$. The right-hand side should be understood as $m(f(u),v) = \int_\Om f_\Om(u)\, v \dx + \int_\Ga f_\Ga(\gamma u)\, \gamma v \d \sigma$.
%
%
\subsection{Splitting methods of \cite{KovL17}}
\label{section:splittings - motivation}
Two splitting methods for parabolic problems with dynamic boundary conditions were proposed in~\cite[Sect.~6.3]{KovL17}. However, both of these methods suffer from some kind of \emph{order reduction}. This issue serves as the main motivation for this paper. 
We briefly recall the two methods: The bulk--surface \emph{force splitting} in~\cite{KovL17} is defined by splitting the bilinear forms into its bulk and boundary integrals as in~\eqref{eq:bilinear form splitting}.
The bulk--surface \emph{component splitting} in~\cite{KovL17} is defined by separating the bulk and boundary components of the matrix--vector formulation.

Convergence experiments in~\cite{KovL17} have demonstrated that their error behaviour (in $L^2$ and $H^1$ norms) is not satisfactory. In Figure~\ref{fig:forceComponent} we present the $L^\infty([0,T], L^2(\Omega;\Gamma))$ norm errors of the splitting schemes. We either observe an order reduction or an~$h$-dependency of the involved error constants. We note here, see~\cite{KovL17}, that the order reduction in the $H^1(\Omega;\Gamma)$ norm is even more pronounced. 
Moreover, the schemes may approximate the solution of a perturbed system, see the discussion in Section~\ref{sec:numerics:Lie:naive} for the corresponding Lie splitting schemes. 
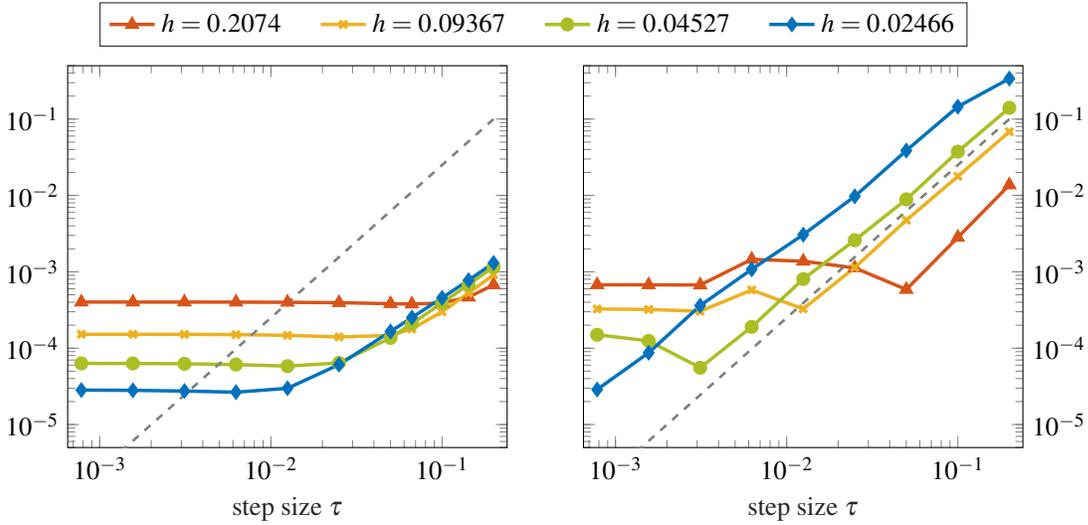
\begin{figure}[htbp]
	\centering
%
%
\begin{tikzpicture}

\begin{axis}[%
width=2.3in,
height=2.0in,
at={(-2.7in,0.in)},
scale only axis,
xmode=log,
xmin=0.00065,
xmax=0.24,
xminorticks=true,
xlabel style={font=\color{white!15!black}},
xlabel={step size $\tau$},
ymode=log,
ymin=5e-06,
ymax=0.5,
yminorticks=true,
ylabel style={font=\color{white!15!black}},
axis background/.style={fill=white},
title style={font=\bfseries},
legend columns = 4,
legend style={legend cell align=left, align=left, at={(1.06,1.05)}, anchor=south, draw=white!15!black}
]

\addplot [color=mycolor1, line width=1.3, mark=triangle*]
table[row sep=crcr]{%
	0.2	0.00067486243\\
	0.142857142857143	0.00046584205\\
	0.1	0.00039054592\\
	0.0666666666666667	0.00037837864\\
	0.05	0.00038257584\\
	0.025	0.0003940563\\
	0.0125	0.00039923465\\
	0.00625	0.00040094999\\
	0.003125	0.00040145341\\
	0.0015625	0.00040159157\\
	0.00078125	0.00040162798\\
};
\addlegendentry{$h=0.2074$ \quad}

\addplot [color=mycolor2, mark=x, line width=1.3]
table[row sep=crcr]{%
	0.2	0.00091639303\\
	0.142857142857143	0.00052579173\\
	0.1	0.00029988237\\
	0.0666666666666667	0.00017879325\\
	0.05	0.00014673215\\
	0.025	0.00014021318\\
	0.0125	0.00014730552\\
	0.00625	0.00015056627\\
	0.003125	0.00015165686\\
	0.0015625	0.00015197991\\
	0.00078125	0.00015206924\\
};
\addlegendentry{$h=0.09367$ \quad}

\addplot [color=mycolor3, line width=1.3, mark=*]
table[row sep=crcr]{%
	0.2	0.0011587936\\
	0.142857142857143	0.00067939234\\
	0.1	0.00039065431\\
	0.0666666666666667	0.00020954866\\
	0.05	0.00013624601\\
	0.025	6.4125305e-05\\
	0.0125	5.8256983e-05\\
	0.00625	6.1048913e-05\\
	0.003125	6.252734e-05\\
	0.0015625	6.3051116e-05\\
	0.00078125	6.3213151e-05\\
};
\addlegendentry{$h=0.04527$ \quad}

\addplot [color=mycolor4, line width=1.3, mark=diamond*]
table[row sep=crcr]{%
	0.2	0.0013061403\\
	0.142857142857143	0.00077636594\\
	0.1	0.00045605957\\
	0.0666666666666667	0.00025178708\\
	0.05	0.00016540965\\
	0.025	6.0932516e-05\\
	0.0125	2.9832861e-05\\
	0.00625	2.6519272e-05\\
	0.003125	2.7461162e-05\\
	0.0015625	2.805132e-05\\
	0.00078125	2.8274343e-05\\
};
\addlegendentry{$h=0.02466$}

\addplot [color=gray, dashed, line width=1.0pt, forget plot]
table[row sep=crcr]{%
	0.2	0.1\\
	0.142857142857143	0.0510204081632653\\
	0.1	0.025\\
	0.0666666666666667	0.0111111111111111\\
	0.05	0.00625\\
	0.025	0.0015625\\
	0.0125	0.000390625\\
	0.00625	9.765625e-05\\
	0.003125	2.44140625e-05\\
	0.0015625	6.103515625e-06\\
	0.00078125	1.52587890625e-06\\
};

\end{axis}


\begin{axis}[%
width=2.3in,
height=2.0in,
scale only axis,
xmode=log,
xmin=0.00065,
xmax=0.24,
xminorticks=true,
xlabel style={font=\color{white!15!black}},
xlabel={step size $\tau$},
ymode=log,
ymin=5e-06,
ymax=0.5,
yticklabel pos=right,
yminorticks=true,
axis background/.style={fill=white},
title style={font=\bfseries},
legend style={at={(0.97,0.03)}, anchor=south east, legend cell align=left, align=left, draw=white!15!black}
]

\addplot [color=mycolor1, line width=1.3, mark=triangle*]
table[row sep=crcr]{%
	0.2	0.013728475\\
	0.1	0.0028226398\\
	0.05	0.00058666427\\
	0.025	0.0011247197\\
	0.0125	0.001375668\\
	0.00625	0.0014573753\\
	0.003125	0.00067167136\\
	0.0015625	0.00067501464\\
	0.00078125	0.00067591516\\
};

\addplot [color=mycolor2, mark=x, line width=1.3]
table[row sep=crcr]{%
	0.2	0.068383836\\
	0.1	0.017790219\\
	0.05	0.0047375882\\
	0.025	0.0011380443\\
	0.0125	0.00032722141\\
	0.00625	0.00057840795\\
	0.003125	0.00030418899\\
	0.0015625	0.00032119829\\
	0.00078125	0.00032613844\\
};

\addplot [color=mycolor3, line width=1.3, mark=*]
table[row sep=crcr]{%
	0.2	0.14023907\\
	0.1	0.037466666\\
	0.05	0.0088835943\\
	0.025	0.0025968242\\
	0.0125	0.000804043\\
	0.00625	0.00019012605\\
	0.003125	5.5472234e-05\\
	0.0015625	0.00012445917\\
	0.00078125	0.00014995903\\
};

\addplot [color=mycolor4, line width=1.3, mark=diamond*]
table[row sep=crcr]{%
	0.2	0.33902809\\
	0.1	0.14535994\\
	0.05	0.038692315\\
	0.025	0.009734111\\
	0.0125	0.0030662945\\
	0.00625	0.0010773354\\
	0.003125	0.00036094181\\
	0.0015625	8.7168718e-05\\
	0.00078125	2.8753095e-05\\
};


\addplot [color=gray, dashed, line width=1.0pt, forget plot]
table[row sep=crcr]{%
	0.2	0.1\\
	0.1	0.025\\
	0.05	0.00625\\
	0.025	0.0015625\\
	0.0125	0.000390625\\
	0.00625	9.765625e-05\\
	0.003125	2.44140625e-05\\
	0.0015625	6.103515625e-06\\
	0.00078125	1.52587890625e-06\\
};

\end{axis}
\end{tikzpicture}%
	\caption{Temporal convergence test with errors in $L^\infty(0,T; L^2(\Omega;\Gamma))$ for force (left) and component splitting (right) applied to the Allen--Cahn equation with a double-well potential. The gray dashed reference line indicates order~$2$. (Image courtesy of the authors of \cite{KovL17}.)}
	\label{fig:forceComponent}
\end{figure}
%
%
\subsection{Formulation as partial differential--algebraic equation}
\label{section:PDAE formulation}
An alternative abstract framework for the weak formulation as coupled system was presented in~\cite{Alt19}. With the introduction of the variable~$p\coloneqq u|_\Gamma$, equation~\eqref{eq:dynamicBC:b} can be written as~$\dot p - \beta\, \Delta_\Gamma p + \partialkn u + \alpha_\Gamma\, p = f_\Gamma(p)$ on $\Gamma$. With this, we can interpret~\eqref{eq:dynamicBC} as two dynamic equations which are coupled through~$p=u|_\Gamma$. As a result, we introduce function spaces of the form 
\[
\V \coloneqq V_u \times V_p , \qquad
\cH \coloneqq H_u \times H_p ,
\]  
with~$V_u \coloneqq H^1(\Omega)$ and~$H_u \coloneqq L^2(\Omega)$ forming a Gelfand triple~$V_u \subseteq H_u \subseteq V_u'$. The function spaces for~$p$ depend on the parameter~$\beta$ and read~$V_p =  H^{\sfrac 12}(\Gamma)$ if $\beta=0$ and~$V_p =  H^{1}(\Gamma)$ otherwise. In both cases we have $H_p = L^{2}(\Gamma)$. 
\begin{remark}
	For~$\beta=0$ also $V_p = L^2(\Gamma)$ is possible. This simplifies the construction of stable finite element schemes. 
\end{remark}
For the formulation of the coupling condition, we define~$\Q \coloneqq H^{-\sfrac 12}(\Gamma)$ and the operator~$\calB\colon \V \to \Q'$ by~$\calB\, \big[\!\begin{smallmatrix} u\\ p \end{smallmatrix}\! \big] \coloneqq p - u|_\Gamma \in \Q'$. Its dual operator is denoted by~$\calB'\colon\Q\to\V'$. The coupling condition is enforced by an additional Lagrange multiplier~$\lambda\colon [0,T] \to \Q$. In operator form, this then leads to the PDAE  
\begin{subequations}
	\label{eq:PDAE}
	\begin{align}
	\begin{bmatrix} \dot u \\ \dot p  \end{bmatrix}
	+ \begin{bmatrix} \calK_\Omega + \alpha_\Omega&  \\  & \beta\calK_\Gamma + \alpha_\Gamma \end{bmatrix}
	\begin{bmatrix} u \\ p  \end{bmatrix}
	+ \calB'\lambda 
	&= \begin{bmatrix} f_\Omega(u) \\ f_\Gamma(p) \end{bmatrix} \qquad \text{in } \V', \\
	\calB\, \begin{bmatrix} u \\ p  \end{bmatrix} \phantom{i + \calB \lambda} &= \phantom{[]} 0\hspace{4.6em} \text{in } \Q'.
	\end{align}
\end{subequations}
Here, the differential operators~$\calK_\Omega\colon V_u\to V_u'$ and~$\calK_\Gamma\colon V_p\to V_p'$ (for $\beta>0$) read
\[
\langle \calK_\Omega u, v\rangle 
:= \int_\Omega \kappa\, \nabla u \cdot \nabla v \dx, \qquad 
\langle \calK_\Gamma p, q\rangle
:= \int_\Gamma \nabla_\Gamma p \cdot \nabla_\Gamma q \dx.   
\]
Since we now have two dynamic variables, namely $u$ and $p$, we have two initial conditions~$u(0) = u^0$ and~$p(0) = p^0$. 
\begin{remark}
	The connection of the two abstract formulations is given by the fact that~$\mathbb{V}$ is isomorphic to the space~$\ker\calB \subseteq \V$; cf.~\cite{Wie19}. 
	This means that, in the first approach, the connection of~$u$ and~$p$ is given a priori in the function space.
\end{remark}
The spatial discretization of the coupled system~\eqref{eq:PDAE} is subject of the following section. 
%
%
\section{Spatial discretization with bulk--surface finite elements}
\label{sec:formulation:space}
For the numerical solution we consider a linear finite element method. Following~\cite{ElliottRanner} and~\cite[Sect.~3.2.1]{KovL17}, we will briefly recall the construction of the discrete domain, the finite element spaces, and the lift operation which can be used to spatially discretize the PDAE~\eqref{eq:PDAE}.
%
%
\subsection{The bulk--surface finite element method}
\label{subsection:bulk surface finite element method}
The domain $\Omega$ is approximated by a triangulation $\calT_h$ with maximal mesh width $h$. 
The union of all elements  of $\calT_h$ defines the polyhedral domain $\Omega_h$ whose boundary $\Gamma_h \coloneqq \pa \Omega_h$ is an interpolation of $\Gamma$, i.e., the vertices of $\Gamma_h$ are on $\Gamma$.
We assume that $h$ is sufficiently small to ensure that for every point $x\in\Gamma_h$ there is a unique point $y\in\Gamma$ such that $x-y$ is orthogonal to the tangent space $T_y\Gamma$ of $\Gamma$ at $y$.
For convergence results, we consider a quasi-uniform family of such triangulations $\calT_h$ of $\Omega_h$; cf.~\cite{ElliottRanner}. For more details we refer to the descriptions in~\cite{ElliottRanner} and~\cite{KovL17}.

We will use the convention that the nodes of the triangulation of $\Omega_h$ are denoted by $(x_k)_{k=1}^\dofOm$, with the number of degrees of freedom~$\dofOm$, while the number of degrees of freedom on~$\Ga_h$ equals $\dofGa \leq \dofOm$. 

The (nonconforming) finite element space $S_h \nsubseteq H^1(\Omega)$ corresponding to $\calT_h$ is spanned by continuous, piecewise linear nodal basis functions on $\Omega_h$, satisfying for each node $(x_k)_{k=1}^\dofOm$
$$
\phi_j(x_k) = \delta_{jk}, \qquad \text{for } j,k = 1, \dotsc, \dofOm .
$$
Then the finite element space is given as
$$
S_h = \textnormal{span}\{\phi_1, \dotsc, \phi_\dofOm \} .
$$
We note here that the restrictions of the basis functions to the boundary $\Gamma_h$ again form a surface finite element basis over the approximate boundary elements.

We define the index sets $\calN_\Om$ and $\calN_\Ga$, of size $\dofOm$ and $\dofGa$, collecting the global numbering of bulk and surface nodes, respectively. We further assume that the nodes are ordered such that the surface nodes are the last $N_\Ga$ in the set $\calN_\Om$.

Following~\cite{Dziuk88}, we define the \emph{lift} of functions $v_h\colon \Gamma_h\to \R$ as 
\begin{equation}
\label{eq:lift definition}
v_h^\ell \colon \Gamma \to \R \quad \text{with} \quad v_h^\ell(y) = v_h(x), 
\end{equation}
for $y \in \Gamma$, where $x\in\Gamma_h$ is the \emph{unique} point on $\Gamma_h$ with $x-y$ being orthogonal to the tangent space $T_y\Gamma$. 
We further consider the \emph{lift} of functions $v_h\colon\Omega_h\to \R$ to $v_h^\ell\colon\Omega\to\R$ by setting $v_h^\ell(y)=v_h(x)$  if $x\in\Omega_h$ and $y \in \Omega$ are related as described in detail in~\cite[Sect.~4]{ElliottRanner}. 
The mapping $G_h\colon \Omega_h \to \Omega$ is defined piecewise, for an element $E \in \calT_h$, by 
\begin{equation}
\label{eq:bulk mapping}
G_h|_E (x) = F_e\big((F_E)^{-1}(x)\big), \qquad \text{for } x \in E.
\end{equation}
Here, $F_e$ is a $C^1$ map (see \cite[eq.~(4.2) \& (4.4)]{ElliottRanner}) from the reference element onto the smooth element $e \subseteq \Omega$ and $F_E$ is the standard affine liner map between the reference element and $E$; see, e.g., \cite[eq.~(4.1)]{ElliottRanner}. 
Note that both definitions of the lift coincide on $\Gamma$. Finally, the lifted finite element space is denoted by $S_h^\ell$, and is given as $S_h^\ell = \{  v_h^\ell \mid v_h \in S_h \}$. 
%
%
\subsection{Matrix--vector formulation of the finite element semi-discretization}
\label{section:matrix-vector formulation}
The bulk--surface finite element discretization of~\eqref{eq:PDAE} results in a DAE of the form
\begin{subequations}
	\label{eq:semidiscreteDAE}
	\begin{align}
	\begin{bmatrix} M_\Omega &  \\  & M_\Gamma \end{bmatrix}
	\begin{bmatrix} \dot u \\ \dot p  \end{bmatrix}
	+ \begin{bmatrix} A_\Omega &  \\  & A_\Gamma \end{bmatrix}
	\begin{bmatrix} u \\ p  \end{bmatrix}
	+ \begin{bmatrix} B^T \\ -M_\Gamma \end{bmatrix} \lambda 
	&= \begin{bmatrix} f_\Omega(u) \\ f_\Gamma(p) \end{bmatrix}, \label{eq:semidiscreteDAE:a} \\[1mm]
	Bu - M_\Gamma p 
	&= 0. \label{eq:semidiscreteDAE:b}
	\end{align}
\end{subequations}
Note that we use the same notion for the semi-discrete variables as for the continuous setting. Moreover, the mesh on the boundary used for the discretization of~$p$ coincides with the bulk-mesh~$\calT_h$ restricted to the boundary. The involved matrices, which correspond to the spatially discrete counterparts of the bilinear forms~$a_\Omega$, $a_\Ga$, $m_\Omega$, and $m_\Ga$, cf.~\eqref{eq:bilinear form splitting}, are given by
\begin{align*}
&\ \begin{aligned}
M_\Omega|_{ij} = m_{\Omega_h}(\phi_j,\phi_i),  \qquad 
A_\Omega|_{ij} = a_{\Omega_h}(\phi_j,\phi_i) 
\end{aligned} 
\qquad\text{for } i,j \in \calN_\Om, \\[2mm]
&\ \begin{aligned}
M_\Gamma|_{ij} = m_{\Gamma_h}(\phi_j,\phi_i),  \qquad 
\ \ A_\Gamma|_{ij} = a_{\Gamma_h}(\phi_j,\phi_i)
\end{aligned} 
\qquad\ \text{for } i,j \in \calN_\Ga.
\end{align*}	
The coupling matrix is defined by 
\begin{align*}	
B|_{ij} = m_{\Gamma_h}(\phi_j,\phi_i) \qquad
\text{for } i \in \calN_\Ga, \ j \in \calN_\Om.
\end{align*}

In the spatially discretized setting, we search for~$u\colon[0,T] \to \R^\dofOm$, $p\colon[0,T] \to \R^\dofGa$, and the Lagrange multiplier~$\lambda\colon[0,T] \to \R^\dofGa$. 
Natural assumptions are that the mass matrices $M_\Omega, M_\Gamma$ are symmetric and positive definite, that the stiffness matrices~$A_\Omega, A_\Gamma$ are symmetric and semi-positive definite, and that~$B \in \R^{\dofGa,\dofOm}$ has full row-rank. As a result, the matrix~$BM_\Omega^{-1}B^T + M_\Gamma \in \R^{\dofGa,\dofGa}$ is invertible and system~\eqref{eq:semidiscreteDAE} is a DAE of index 2. 

We would like to emphasize that the considered PDAE model of parabolic problems with dynamic boundary conditions allows different discretizations of $p$ and $u$ on the boundary, cf.~\cite{Wie19,AltV21}. In this paper, however, we restrict ourselves to the case where the degrees of freedom for $p$ coincide with the degrees of freedom of $u$ on the boundary. 
In view of the definition of the index sets $\calN_\Om$ and $\calN_\Ga$, the bulk and surface components of $u$ are ordered in such a way that the last~$\dofGa$ components correspond to the boundary. 
Thus, the constraint matrix has the particular form~$B = [\, 0\ M_\Gamma] \in \R^{\dofGa \times \dofOm}$. 
This also allows a convenient decomposition of the bulk variable, namely 
\begin{align}
\label{eq:u1u2}
u = \begin{bmatrix}
u_1 \\ u_2
\end{bmatrix}
\qquad\text{with }
u_1(t)\in\R^{\dofOm-\dofGa},\ u_2(t)\in\R^{\dofGa}.
\end{align}
Analogously, we decompose $f_\Omega$ into $f_1$ and~$f_2$ and~$M_\Omega$, $A_\Omega$ into the blocks $M_{ij}$, $A_{ij}$, $i,j=1,2$.
%
%
\subsection{Shape regularity and a boundary estimate}
For the convergence analysis presented in Section~\ref{sec:Lie:convergence}, we assume that the triangulation~$\calT_h$ is \emph{quasi-uniform}, i.e., there exists a constant $\vartheta\ge 2$ such that 
\begin{equation*}
\frac{r}{R} \geq \vartheta^{-1}
\end{equation*}
holds for every element $E$. Here, $R$ and $r$ denote the circumradius and inradius of $E$, respectively. 

Under this assumption, we will now show the following key estimate concerning the $M_{22}$ and the $M_\Gamma$ norms.
\begin{lemma}\label{lem_cM}
	Let $\calT_h$ be a quasi-uniform triangulation with~$h_{\Gamma}$ denoting the maximal mesh width on the discrete boundary~$\Gamma_h$. Then there exists a constant $c_M >0$, which only depends on the uniformity parameter~$\vartheta$ but is independent of~$h_\Gamma$ (the maximal mesh width on the boundary $\Ga_h$) and~$h$, such that 
	\[
	\| \cdot \|_{M_{22}}^2 
	\leq c_M h_\Gamma\, \| \cdot \|_{M_{\Gamma}}^2. 
	\]
\end{lemma}
\begin{proof}
	We prove the two-dimensional case in detail and assume for the sake of readability that every element of~$\calT_h$ has at most one boundary edge. The proof for $d=3$ is similar and we only comment on the necessary modifications.
	\medskip
	
	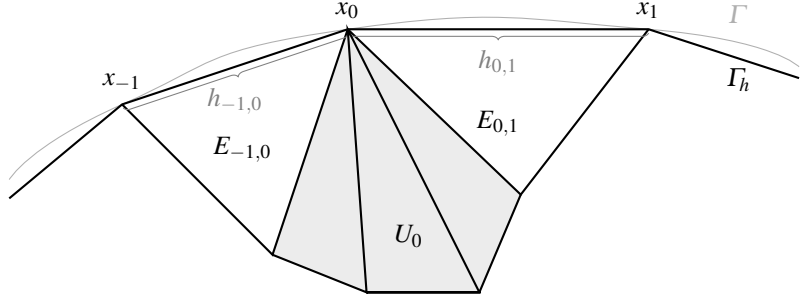
\begin{figure}[htbp]
	\centering
		\begin{tikzpicture}[scale=1]
		\coordinate (vm1) at (-3,-1);
		\coordinate (v0) at (0,0);
		\coordinate (v1) at (4,0);
		\coordinate (p1) at (0.25,-3.5);
		\coordinate (p2) at (1.75,-3.5);
		\fill[black!7.5!white] (-1,-3) -- (p1) -- (p2) -- (2.3,-2.2) -- (v0);
		\draw [lightgray] plot [smooth, tension=1] coordinates { (-4.5,-2) (vm1) (v0)  (v1) (6,-0.5)};
		\node[lightgray] at (5.2,0.2) {{$\Ga$}};
		\draw [black,thick] (-4.5,-2.25) -- (vm1) -- (v0) -- (v1) -- (6,-0.65);
		\node[black, below] at (5.2,-0.4) {{$\Ga_h$}};
		\node[black,yshift=0.3cm] at (vm1) {$x_{-1}$};
		\node[black,yshift=0.25cm] at (v0) {$x_{0}$};
		\node[black,yshift=0.25cm] at (v1) {$x_{1}$};
		\draw [black,thick] (vm1) -- (-1,-3) -- (v0); 
		\draw [black,thick] (v0) -- (2.3,-2.2) -- (v1); 
		\draw [gray,decoration={brace,mirror,raise=0.05cm},decorate] (vm1) -- (v0) 
		node [pos=0.5,anchor=north,yshift=-0.15cm] {$h_{-1,0}$};
		\draw [gray,decoration={brace,mirror,raise=0.05cm},decorate] (v0) -- (v1) 
		node [pos=0.5,anchor=north,yshift=-0.15cm] {$h_{0,1}$};
		\node[black,yshift=0.25cm] at (-1.4,-1.85) {$E_{-1,0}$};
		\node[black,yshift=0.25cm] at (2,-1.5) {$E_{0,1}$};
		\draw [black,thick] (-1,-3) -- (p1) -- (v0) -- (p2) -- (p1) -- (p2) -- (2.3,-2.2); 
		
		\node[black] (U0) at (0.5,-2.75) [right] {$U_0$};
		\end{tikzpicture}
		\caption{Notation and construction for the proof of Lemma~\ref{lem_cM}.}
		\label{fig:boundary_estimate_plot}
	\end{figure}  
	
	\emph{Case $d=2$:}  Let $E \in \calT_h$ be an arbitrary triangle of the triangulation with circumradius~$R$ and inradius~$r$. By~\cite{Lon03} we have $1 + r/R = \sum_{i=1}^3 \cos \alpha_i$ for the angle $\alpha_i$, $i=1,\ldots,3$, of $E$. It is easy to show that $\alpha \coloneqq \alpha_1$ is an extremum only if the other two angles are equal, i.e.,  
	\begin{equation}
	\label{eqn_inequality_elements_angle}
	\vartheta^{-1} \leq \tfrac{r}{R} = \cos \alpha + 2 \cos \tfrac{\pi-\alpha}{2} -1 = 2 \sin \tfrac{\alpha}{2} - 2 \sin^2 \tfrac{\alpha}{2}.
	\end{equation}
	Thus, every angle in $E$ is bounded from below by $2\arcsin(\tfrac{1}{2}-\tfrac{1}{2}\sqrt{1-2\vartheta^{-1}})$.   
	Furthermore, every pair of sides of $E$ with lengths $a$, $b$ satisfies 
	\begin{equation}
	\label{eqn_inequality_elements_sides}
	a \leq 2 R \leq 2 \vartheta r \leq \vartheta b.
	\end{equation}
	This implies for the area of $E$, 
	\begin{equation*}
	A_E = 2 R^2 {\textstyle \prod_{i=1}^3} \sin \alpha_i \leq \tfrac{3\sqrt{3} }{4} R^2 \leq \tfrac{3\sqrt{3} }{16}\vartheta^2 a^2.
	\end{equation*}
	
	For a node~$x_0 \in \Gamma_h$, let~$x_{-1}$ and~$x_{1}$ denote the two unique nodes which share an edge on the discrete boundary~$\Gamma_h$ with~$x_0$. For~$x_0$ and~$x_1$, the length of the edge between them is~$h_{0,1}$ and~$E_{0,1}$ is the unique triangle containing these nodes. 
	Analogously, we set~$h_{-1,0}$ and~$E_{-1,0}$ for~$x_{-1}$, $x_0$; cf.~the sketch in Figure~\ref{fig:boundary_estimate_plot}. 
	Finally, $\varphi_{0}$ equals the nodal basis function associated to the node~$x_0$ and analogously~$\varphi_{-1}$ and~$\varphi_{1}$. 
	Then, for arbitrary $z_0, z_1 \in \R$ we observe 
	\begin{align}
	\label{eqn_estimate_E01}
	\int_{E_{0,1}} \!(z_{0}\varphi_{0} + z_{1}\varphi_{1})^2 \dx \leq \tfrac{1}{3}(z_{0}^2 + z_{1}^2) A_{E_{0,1}} \leq \tfrac{\sqrt{3}\vartheta^2}{16} h_{0,1}^2 (z_{0}^2 + z_{1}^2) \leq \hat{c}(\vartheta) h_{\Gamma}\, h_{0,1} (z_{0}^2 + z_{1}^2).
	\end{align}
	A similar estimate holds for~$E_{-1,0}$. 
	
	Let us now denote the union of all elements of $\calT_h$ containing $x_0$ without~$E_{-1,0}$ and~$E_{0,1}$ by~$U_0$. For an estimate of $z_0^2 \varphi_0^2$ restricted to $U_0$, we use a similar approach as in~\eqref{eqn_estimate_E01} and estimate $\sum_{E \in U_0} A_E$. Note that~$U_0$ contains at most $n_0 \coloneqq \lfloor \pi/\arcsin(\tfrac{1}{2}-\tfrac{1}{2}\sqrt{1-2\vartheta^{-1}}) \rfloor -2$ elements. 
	By the estimate~\eqref{eqn_inequality_elements_sides}, the element $\widetilde{E} \in U_0$ which shares a side with~$E_{0,1}$ has two sides with at most length~$\vartheta h_{0,1}$ and~$\vartheta^2 h_{0,1}$ and thus, $A_{\widetilde{E}} \leq \frac 1 2 \vartheta^3 h_{0,1}^2 \sin \alpha \leq \frac 1 2 \vartheta^3 h_{0,1}^2$. By considering sequentially the elements of~$U_0$ such that the current and the previous element share an edge, we get 
	\begin{align}
	\label{eqn_estimate_U0}
	\int_{U_0} z_0^2 \varphi^2 \dx = \tfrac{z_0^2}{6} \sum_{E \in U_0} A_{E} \leq \tfrac{\vartheta^3 z_0^2}{12} (h_{-1,0}^2 + h_{0,1}^2) \sum_{\ell =0}^{\lceil n_0/2 \rceil} \vartheta^{2\ell} \leq \tilde{c}(\vartheta) h_{\Gamma}\, (h_{-1,0} + h_{0,1}) z_0^2.  
	\end{align} 
	Finally, we obtain with the estimates~\eqref{eqn_estimate_E01} and~\eqref{eqn_estimate_U0} for an arbitrary discrete $u$, 
	\begin{align*}
	\|u\|_{M_{22}}^2 & =  \sum_{x_0 \in \Gamma_h} \tfrac 12 \int_{E_{-1,0}} u^2 \dx + \int_{U_{0}} u^2 \dx + \tfrac 12 \int_{E_{0,1}} u^2 \dx\\
	& \leq  h_{\Gamma} \sum_{x_0 \in \Gamma_h} \tfrac{\hat{c}(\vartheta)}{2} h_{-1,0} u^2(x_{-1})  +  \big(\tfrac{\hat{c}(\vartheta)}{2} + \tilde{c}(\vartheta)\big) (h_{-1,0} + h_{0,1}) u^2(x_0) + \tfrac{\hat{c}(\vartheta)}{2} h_{0,1} u^2(x_{1})\\
	& \leq 6\, \big(\tfrac{\hat{c}(\vartheta)}{2} + \tilde{c}(\vartheta)\big)\, h_{\Gamma}\, \sum_{v_0 \in \Gamma_h} \int_{\Gamma_h \cap E_{-1,0}} u^2 \d \sigma +  \int_{\Gamma_h \cap E_{0,1}} u^2 \d \sigma \\
	& = 12\, \big(\tfrac{\hat{c}(\vartheta)}{2} + \tilde{c}(\vartheta)\big)\, h_{\Gamma}\, \|u\|_{M_{\Gamma}}^2.
	\end{align*}
	
	\emph{Case $d=3$:} In three dimensions, the statement can be proven similarly. For this, we note that every face of an element~$E$ of $\calT_h$ is also quasi-uniform with the same~$\vartheta$. In particular, by~\cite[Thm.~4.1]{MinP08}, the two sides with area $A_1$ and $A_2$ satisfy
	\begin{equation*}
	A_1 \leq \tfrac{3\sqrt{3} }{4} R^2\leq \tfrac{3\sqrt{3} }{4} \vartheta^2 r^2 \leq \tfrac{\vartheta^2}{4} A_2.
	\end{equation*}
	Furthermore, the volume of $E$ is bounded by 
	\begin{equation*}
	V_E \leq \tfrac{3\pi}{4} R^3 \leq \tfrac{3 \pi}{4} \vartheta^3  r^3 \leq  \tfrac{\pi}{12 \sqrt[4]{3}} \vartheta^3  A^{3/2}
	\end{equation*} 
	with $A$ being the area of an arbitrary side of $E$. If this side is part of~$\Gamma_h$, then we have $A^{3/2} \leq \sqrt[4]{27}/4 \cdot h_{\Gamma} A$. Finally, by solid angles, one can show that every node has at most 
	\begin{equation*}
	n_0 \coloneqq \Big\lfloor \pi / \arctan\Big( (2 - \sqrt 3)\sqrt{\tfrac{\sqrt{3}- 2 \sin \alpha/2}{\sqrt{3}+ 2 \sin \alpha/2}}\Big)\Big \rfloor 
	\end{equation*}
	elements containing this node. The angle~$\alpha$ is again bounded by~\eqref{eqn_inequality_elements_angle}.
\end{proof}
\begin{remark}
	The estimate of the constant $c_M$ in the proof of Lemma~\ref{lem_cM} is rather pessimistic. For example, the proven constant $c_M$ for a criss-cross triangulation, i.e., $\vartheta = 1 + \sqrt 2$, of a unit square is~$577.9$
	, whereas numerical tests show that $c_M \approx 0.3$. 
\end{remark}
%
%
\section{Bulk--surface Lie splitting}
\label{sec:Lie}
This section is devoted to a first-order splitting approach, normally referred to as \emph{Lie splitting}. 
For this, we need to identify two subsystems, which are then solved on small time intervals of length~$\tau$ in an alternating manner. 
Here, the idea is to split the dynamics in the bulk and on the boundary. 
This means that one of the subsystems is a pure boundary problem and thus, of small (spatial) dimension.  

As first subsystem we consider~\eqref{eq:semidiscreteDAE} without the dynamic equation for $p$, i.e., we consider the DAE 
\begin{subequations}
	\label{eq:sub1:bulk}
	\begin{align}
	M_\Omega \dot u + A_\Omega u + B^T \lambda 
	&= f_\Omega(u), \\
	Bu \phantom{+ MB \lambda}
	&= M_\Gamma p .
	\end{align}
\end{subequations}
This system equals the DAE formulation of a parabolic problem, where the (inhomogeneous) Dirichlet boundary conditions are included in form of an explicit constraint, cf.~\cite{Alt15}. 
Here, however, the boundary data is given by (the unknown) $p$.

The second subsystem is then a pure boundary problem. 
The second line of~\eqref{eq:semidiscreteDAE:a} reads~$M_\Gamma \dot p + A_\Gamma p = f_\Gamma(p) + M_\Gamma \lambda$ and inserting the Lagrange multiplier (i.e., the dynamic equation for~$u_2$) leads to
%
\begin{align}
\label{eq:sub2:boundary}
M_\Gamma \dot p + A_\Gamma p 
= f_\Gamma(p) + f_{2}(u)
- M_{22} {\dot u}_2 - A_{22} u_2
- M_{21} {\dot u}_1 - A_{21} u_1,
\end{align}
where $u_1, u_2$ again denote the interior and boundary part of $u$ as introduced in~\eqref{eq:u1u2}. 
Note that we do \emph{not} apply the equation $p=u_2$, since we distinguish here $p$ as the unknown and $u_2$ as the input coming from the previous subsystem. Further, we have used the special structure of~$B$ as discussed in Section~\ref{section:matrix-vector formulation}. 

We now consider different splitting approaches -- that differ from those of Section~\ref{section:splittings - motivation}: Starting with a naive approach, which will not yield the desired first-order convergence. Afterwards, we apply the splitting to a regularized formulation, which will yield the expected first-order convergence rates.
%
%
\subsection{Failure of naive PDAE approach}\label{sec:Lie:naive}
We first consider the direct application of Lie splitting for the two subsystems~\eqref{eq:sub1:bulk} and~\eqref{eq:sub2:boundary}. 
For this, we consider the interval $[0,\tau]$ with given initial data $\psinit=\usAinit$ and $\usDinit$.  

The first subsystem, which considers the bulk, comes together with the equation $\dot p = 0$. This means that we freeze $p$ at time $t=0$, leading to 
\begin{align*}
M_\Omega \dot u + A_\Omega u + B^T \lambda 
&= f_\Omega(u), \\
Bu \phantom{+ MB \lambda}
&= M_\Gamma \psinit ,
\end{align*}
with initial condition $u(0) = (\usDinit , \usAinit)^T$. 
Hence, we solve the bulk problem with time-independent Dirichlet boundary conditions. 
The outcome are the functions $u_2\equiv \psinit$ and, since $\dot u_2 = 0$, $u_1$ as the solution of   
\[
M_{11}{\dot u}_1 + A_{11} u_1 
= f_{1}(u) - A_{12} \psinit. 
\]

Then, in the second step, we consider the boundary problem with initial data $\psinit=\usAinit$ and $u_1(\tau)$ from the first subsystem. This goes along with $\dot u_1 = 0$ such that~\eqref{eq:sub2:boundary} leads to  
\[
M_\Gamma \dot p + A_\Gamma p 
= f_\Gamma(p) + f_{2}(u) - A_{22} \psinit
- A_{21} u_1(\tau) ,
\]
with initial condition~$p(0) = \psinit$. 

As shown in the numerical experiments of Section~\ref{sec:numerics:Lie:naive}, the presented naive approach of simply fixing the boundary data on a subinterval of length $\tau$ does not yield satisfactory convergence results. It turns out, that it is advisable to include further information of $p$ into the first subsystem, namely its derivative. 
%
%
\subsection{Lie splitting on continuous level}\label{sec:Lie:cont}
In order to include information on the derivative of the boundary data to the bulk system, we consider a reformulation of the bulk problem. This then leads to a splitting scheme, which we again consider on a single subinterval~$[0,\tau]$. As given initial data we consider now $\psinit$, $\dpsinit$, and $\usDinit$.

To include the derivative of $p$ to the first subsystem~\eqref{eq:sub1:bulk} in the right manner, we apply an \emph{index reduction} method known from DAE theory. 
More precisely, we apply \emph{minimal extension}, which yields an extended but equivalent system; see~\cite{MatS93, KunM06, AltH18}. 
The idea is to introduce a new (dummy) variable $w \coloneqq \dot u_2$ and include the derivative of the constraint to the system equations, i.e., 
\begin{align*}
\begin{bmatrix} M_{11} & M_{12} \\ 
M_{21} & M_{22} \end{bmatrix} 
\begin{bmatrix} \dot u_1 \\ w \end{bmatrix}
+ \begin{bmatrix} A_{11} & A_{12} \\ 
A_{21} & A_{22} \end{bmatrix} 
\begin{bmatrix} u_1 \\ u_2 \end{bmatrix}
+ \begin{bmatrix} 0 \\ M_\Gamma \end{bmatrix} \lambda 
&= \begin{bmatrix} f_{1}(u) \\ f_{2}(u) \end{bmatrix}, \\
u_2 &= p, \\
w &= \dot p ,
\end{align*}
with initial condition $u_1(0) = \usDinit$. 
We emphasize that the connection $w = \dot u_2$ is only implicitly part of the equations. 
Further, $u_1$ is the only remaining 'differential' variable such that an initial condition for $u_1$ is sufficient.
In other words, the consistency condition on $u_2$ is directly encoded in the system equations. 

In order to obtain a solvable system (recall that $p$ and $\dot p$ are not yet known on the interval $(0,\tau]$), we again freeze the values of $p$ at the initial time. 
In contrast to the previous approach, however, this means that the derivative is not set to zero but fixed by some value~$\dpsinit$. 
Altogether, this means that the first subsystem of the splitting scheme reads 
\begin{subequations}
	\label{eq:Lie}
	\begin{align}
	\label{eq:Lie:a}
	M_{11}{\dot u}_1 + A_{11} u_1 
	= f_{1}(u) - M_{12}\dpsinit - A_{12} \psinit ,
	\end{align}
	with initial condition $u_1(0) = \usDinit$.  
	Further we get $u_2 \equiv \psinit$ and $\dot u_2 \equiv \dpsinit$, since the boundary data remains untouched within the bulk system.  
	
	As a second step, we insert the solution of the first subsystem into~\eqref{eq:sub2:boundary}. Hence, we need to solve the system 
	\begin{align}
	\label{eq:Lie:b}
	M_\Gamma \dot p + A_\Gamma p 
	= f_\Gamma(p) + f_{2}(u) 
	- M_{22} \dpsinit - A_{22} \psinit
	- M_{21} {\dot u}_1 - A_{21} u_1
	\end{align}
\end{subequations}
in~$[0,\tau]$, with initial condition~$p(0) = \psinit$. 
Note that the outcome of~\eqref{eq:Lie:a} included the constant functions~$u_2$ and~$\dot u_2$, leading to the terms~$M_{22} \dpsinit$ and~$A_{22} \psinit$ in the second subsystem. 

System~\eqref{eq:Lie} defines the \emph{continuous} version of Lie splitting, i.e., we assume here that the two subsystems are solved exactly. 
The here characterized first step of the iteration is then continued for the intervals $[k\tau, (k+1)\tau]$ for $k=1,2,\dotsc,T/\tau$. As initial data on the respective interval, we always consider the final values of $u_1$ and $p$ from the previous interval. 

In the following, we discuss a \emph{fully discrete} Lie splitting, which occurs by an additional (temporal) discretization of the two subsystems. 
%
%
\subsection{Fully discrete Lie splitting}\label{sec:Lie:implEuler}
Since Lie splitting is expected to yield a first-order scheme, we discretize the subsystems by a suitable first-order method. More precisely, we apply the implicit Euler scheme to both subsystems. 
For this, we will denote time-discrete backward differences by
\begin{equation*}
\partial_\tau w^{n+1} 
\coloneqq \tau^{-1}\, \big( w^{n+1} - w^n \big), \qquad n \geq 0 .
\end{equation*}
We will also use the shorthand notation $f_\Gamma^{n+1} = f_\Gamma(p^{n+1})$ and 
$$f_{k}^{n+1} = f_{k}(u_1^{n+1}, p^{n}), \qquad k=1,2.$$

On the time interval $[t^n, t^{n+1}]$ of length $\tau$, a fully discrete Lie splitting step reads as follows: Given~$u_1^n$, $p^n$, and $\partial_\tau p^n$ as approximations of $u_1(t^n)$, $p(t^n)$, and $\dot p(t^n)$, respectively, solve  
\begin{subequations}
	\label{eq:LieEuler}
	\begin{align}
	\label{eq:LieEuler:a}
	M_{11} \partial_\tau u_1^{n+1} + A_{11} u_1^{n+1} 
	= f_{1}^{n+1} - M_{12}\partial_\tau p^n - A_{12} p^n,
	\end{align}
	which yields~$u_1^{n+1}$ and hence, also gives $\partial_\tau u_1^{n+1}$. With these values, then solve in a second step,
	\begin{align}
	\label{eq:LieEuler:b}
	M_\Gamma \partial_\tau p^{n+1} + A_\Gamma p^{n+1} 
	= f_\Gamma^{n+1} + f_{2}^{n+1} - M_{22} \partial_\tau p^n - A_{22} p^n - M_{21} \partial_\tau u_1^{n+1} - A_{21} u_1^{n+1}. 
	\end{align}
\end{subequations}
which yields the updated approximations $p^{n+1}$ and $\partial_\tau p^{n+1}$, used as initial values for the next step. 

In matrix form, these two steps can be written as
\begin{align}
\label{eq:LieEuler_matrix}
&\begin{bmatrix}
M_{11} & 0 \\ M_{21} & M_{\Gamma}
\end{bmatrix}
\begin{bmatrix}
\partial_\tau u_1^{n+1} \\ \partial_\tau p^{n+1}
\end{bmatrix}
+
\begin{bmatrix}
A_{11} & 0 \\ A_{21} & A_{\Gamma}
\end{bmatrix}
\begin{bmatrix}
u_1^{n+1} \\ p^{n+1}
\end{bmatrix} 
+
\begin{bmatrix}
A_{12} \\  A_{22} 
\end{bmatrix}
p^{n}
+
\begin{bmatrix}
M_{12}\\ M_{22}
\end{bmatrix}
\partial_\tau p^{n} 
= 
\begin{bmatrix}
f_1^{n+1} \\ f_2^{n+1} + f_\Gamma^{n+1}
\end{bmatrix}.
\end{align}
\begin{remark}
	Due to the appearance of $\partial_\tau p^n$ in \eqref{eq:LieEuler:a}, iteration~\eqref{eq:LieEuler} (respectively iteration~\eqref{eq:LieEuler_matrix}) is actually a two-step scheme. We may introduce $q\coloneqq\dot p$ as dummy variable in order to get an equivalent formulation as a one-step scheme. 
	On the other hand, the proposed Lie splitting may be interpreted as a time shift of certain terms in the iteration matrices. 
\end{remark}
%
%
\subsection{Convergence analysis}\label{sec:Lie:convergence}
We restrict our attention to the convergence analysis of the \emph{linear} case. The modifications required for the analysis of semi-linear problems is presented in~\ref{app:semi-linear}.

Before discussing the stability of the fully discrete Lie splitting scheme~\eqref{eq:LieEuler}, we recall the so-called \emph{inverse estimate} for finite elements~\cite[Ch.~II.6.8]{Bra07}, which reads in present setting 
\begin{align}
\label{eq:inverseEstimate}
\|\cdot\|_{A_{22}}^2 
\leq \big(\cAM + c_\text{inv}h^{-2} \big)\, \|\cdot\|_{M_{22}}^2. 
\end{align}

Next, we introduce a \emph{weak} CFL condition. Recall that classical CFL conditions in the context of parabolic problems read~$\tau \leq c h^2$, which is a very restrictive assumption of the step size.  

\begin{assumption}[Weak CFL condition]
	\label{ass:CFL}
	The time step size is sufficiently small in the sense that~$(7\, c_\text{inv}c_M)\, \tau < 3\, h$, where $c_M$ is the constant from Lemma~\ref{lem_cM}.  
\end{assumption}
As preparation for the convergence proof of the fully discrete Lie splitting, we consider the following stability result. 

\begin{lemma}[Stability of Lie splitting]
	\label{lem_stability_Lie_impl_Euler}
	Consider a quasi-uniform triangulation and let Assumption~\ref{ass:CFL} be valid. Then the fully discrete Lie splitting scheme~\eqref{eq:LieEuler} is stable in the sense that 
	\begin{multline}
	\|u^{n+1}\|_{A}^2  + \| p^{n+1}\|_{A_{\Gamma}}^2 
	+ \tau \sum_{k=0}^n (1 - 4 c_M h - c_{\alpha,M} \tau h - c_A \tau h^{-1})\big\| \partial_\tau p^{k+1}\big\|_{M_{\Gamma}}^2 
	\\*
	\leq \|u^0\|_{A}^2 + \|p^0\|_{A_{\Gamma}}^2 + 2\tau \|\partial_\tau p^0\|_{M_{22}}^2 + \tau \sum_{k=0}^n  \|f^{k+1}_\Omega\|_{M^{-1}}^2 + \|f^{k+1}_\Gamma\|_{M^{-1}_\Gamma}^2,  \label{eqn_est_Lie_impl_Euler}
	\end{multline}
	where $c_A\coloneqq c_M c_\text{inv} >0$ and $c_{\alpha,M}\coloneqq c_M \cAM \geq 0$ are independent of~$h$ and~$\tau$.  
\end{lemma}
\begin{proof}
	For the following proofs a key technical idea is the following: 
	By the matrix form~\eqref{eq:LieEuler_matrix}, we can write the Lie splitting scheme as a perturbation of the implicit Euler method applied to system~\eqref{eq:semidiscreteDAE} (upon eliminating $\lambda$), i.e., 
	\begin{align}
	&\ \begin{bmatrix}
	M_{11} & M_{12} \\ M_{21} & M_{22} + M_{\Gamma}
	\end{bmatrix}
	\begin{bmatrix}
	\partial_\tau u_1^{n+1} \\ \partial_\tau p^{n+1}
	\end{bmatrix}
	+
	\begin{bmatrix}
	A_{11} & A_{12} \\ A_{21} & A_{22} + A_{\Gamma}
	\end{bmatrix}
	\begin{bmatrix}
	u_1^{n+1} \\ p^{n+1}
	\end{bmatrix}
	\label{eqn_modified_Lie_impl_Euler_help}\\*
	&\qquad\quad 
	+
	\begin{bmatrix}
	M_{12}\\ M_{22}
	\end{bmatrix}
	\big(
	\partial_\tau p^{n} - \partial_\tau p^{n+1}
	\big)
	+
	\begin{bmatrix}
	A_{12} \\ A_{22} 
	\end{bmatrix}
	\big(
	p^{n} - p^{n+1}
	\big)
	= 
	\begin{bmatrix}
	f_1^{n+1} \\ f_2^{n+1} + f_\Gamma^{n+1}
	\end{bmatrix}. \notag
	\end{align}
	Therefore, testing with $\tau \partial_\tau u^{n+1}$, and using that $u_2^{n+1} = p^{n+1}$, we obtain 
	\begin{align*}
	&\ \tau\, \|\partial_\tau u^{n+1}\|_{M}^2 + \tau\, \|\partial_\tau p^{n+1}\|_{M_\Gamma}^2 + \tfrac 12\, \big(\|u^{n+1}\|_{A}^2 - \|u^{n}\|_{A}^2 + \tau^2 \|\partial_\tau u^{n+1}\|_{A}^2 \big)\\
	&\ + \tfrac 12\, \big(\|p^{n+1}\|_{A_\Gamma}^2 - \|p^{n}\|_{A_\Gamma}^2 + \tau^2 \|\partial_\tau p^{n+1}\|_{A_\Gamma}^2 \big)\\
	= &\ \tau\, \Big\langle f^{n+1}_\Omega 
	+ \begin{bmatrix}
	M_{12}\\ M_{22}
	\end{bmatrix}
	[\partial_\tau p^{n+1} - \partial_\tau p^{n}]
	+ \tau \begin{bmatrix}
	A_{12} \\ A_{22} 
	\end{bmatrix}
	\partial_\tau p^{n+1},
	\partial_\tau u^{n+1} \Big \rangle + \tau\, \langle f^{n+1}_\Gamma, \partial_\tau p^{n+1} \rangle\\
	\leq &\ \tfrac{\tau}{2}\, \| f_\Omega^{n+1} \|_{M^{-1}}^2 + \tau\, \|\partial_\tau p^{n+1}\|_{M_{22}}^2+ \tau\, \|\partial_\tau p^{n}\|_{M_{22}}^2 + \tau\, \|\partial_\tau u^{n+1}\|_{M}^2\\
	&\ + \tfrac{\tau^2}{2}\, \|\partial_\tau p^{n+1} \|_{A_{22}}^2 + \tfrac {\tau^2}2\, \|\partial_\tau u^{n+1}\|_{A}^2  + \tfrac \tau 2\, \|f_\Gamma^{n+1}\|_{M_\Gamma^{-1}}^2 + \tfrac \tau 2\, \|\partial_\tau p^{n+1}\|_{M_\Gamma}^2.
	\end{align*}
	Summing up this inequality from $0$ to $n$ yields the desired estimate~\eqref{eqn_est_Lie_impl_Euler}, where we use the inverse estimate~\eqref{eq:inverseEstimate} and Lemma~\ref{lem_cM}, leading to~$\|\cdot\|_{A_{22}}^2 \leq (c_{\alpha,M} h + c_Ah^{-1})\|\cdot\|_{M_\Gamma}^2$.
\end{proof}

In the following lemma, we estimate the local error caused by the first step of the Lie splitting scheme. 
\begin{lemma}
	\label{lem_local_error_Lie_impl_Euler}
	Let Assumption~\ref{ass:CFL} be satisfied, $\calT_h$ a quasi-uniform triangulation, and~$\partial_\tau p^0 = \dot{p}(0)$. Then we have
	\begin{multline*}
	\| u^1 - u^0\|_{A}^2 + \| p^1 - p^0\|_{A_\Gamma}^2 + \tau\, \big(1- c_M h - 2 c_{\alpha,M} \tau h  -2 c_A \tau h^{-1}\big)\, \| \partial_\tau p^1 - \dot p(0)\|_{M_\Gamma}^2\\
	\leq \tau^2\, \|\dot{u}(0)\|_{A}^2 + \tau^2\, \|\dot{p}(0)\|_{A_{\Gamma}}^2 + 2\tau^2\, \|\dot{p}(0)\|_{A_{22}}^2
	+ \tau\, \| f_\Omega^1 - f_\Omega^0\|_{M^{-1}}^2 + \tau\, \| f_\Gamma^1 - f_\Gamma^0\|_{M^{-1}_\Gamma}^2.
	\end{multline*}
\end{lemma}
\begin{proof}
	Since $u^0$, $p^0$, and $\partial_\tau p^0$ are the exact initial values, we note that 
	\begin{multline*}
	\begin{bmatrix}
	M_{11} & M_{12} \\ M_{21} & M_{22}+M_{\Gamma}
	\end{bmatrix}
	\begin{bmatrix}
	\partial_\tau u_1^{1} - \dot{u}_1(0) \\ \partial_\tau p^{1} - \dot{p}(0)
	\end{bmatrix}
	+
	\begin{bmatrix}
	A_{11} & A_{12} \\ A_{21} & A_{22} + A_{\Gamma}
	\end{bmatrix}
	\begin{bmatrix}
	u_1^{1} - u_1^0 \\ p^{1} - p^0
	\end{bmatrix}\\
	- 
	\begin{bmatrix}
	M_{12}\\ M_{22}
	\end{bmatrix} \big( \partial_\tau p^{1} - \dot{p}(0)\big)
	-
	\begin{bmatrix}
	A_{12}\\ A_{22}
	\end{bmatrix}
	\big(p^{1} - p^0- \tau \dot p(0)\big) = 
	\begin{bmatrix}
	f_{1}^1 - f_{1}^0 \\f_{2}^1 - f_{2}^0 + f_\Gamma^{1}  - f_\Gamma^{0}  
	\end{bmatrix}
	+ \tau
	\begin{bmatrix}
	A_{12}\\ A_{22}
	\end{bmatrix}
	\dot p(0)
	\end{multline*}
	holds. Testing this equation with $u^1 - u^0  - \tau \dot u(0)$, we get similarly as in the proof of Lemma~\ref{lem_stability_Lie_impl_Euler} the estimate
	\begin{align*}
	&\ \tau\, \big(1- c_M h - 2c_{\alpha,M} \tau h  - 2c_A \tau h^{-1}\big) \| \partial_\tau p^1- \dot p(0)\|_{M_\Gamma}^2 \\
	&\ + \|u^1-u^0\|_A^2 - \|\tau  \dot u(0)\|_A^2 +\|p^1-p^0\|_{A_\Gamma}^2 - \|\tau \dot p(0)\|_{A_\Gamma}^2 + \|p^1-p^0 - \tau \dot p(0)\|_{A_\Gamma}^2 \\
	\leq &\ \tau\, \| f_\Omega^1 - f_\Omega^0\|_{M^{-1}}^2 + \tau\, \| f_\Gamma^1 - f_\Gamma^0\|_{M^{-1}_\Gamma}^2 + 2\, \|\tau \dot p(0)\|_{A_{22}}^2.
	\end{align*}
	This concludes the proof.
\end{proof}

Finally, we are in the position to prove first-order convergence of the proposed Lie splitting scheme provided the assumptions on the spatial and temporal discretization parameters are satisfied. 
\begin{theorem}
	\label{th_error_Lie_impl_Euler}
	Let the assumptions of Lemma~\ref{lem_local_error_Lie_impl_Euler} be satisfied 
	as well as~$7\,h < 1/c_M$. 
	Then, we have
	\begin{subequations}
		\begin{align}
			\label{th_error_Lie_impl_Euler_a} \nonumber &  \| u(t^{n+1})-u^{n+1}\|_{M}^2  + \tau\, {\textstyle \sum_{k=0}^n} \| u(t^{k+1})-u^{k+1}\|_{A_{\hphantom{\Gamma}}}^2\\*
			& \tag{a}\qquad + \| p(t^{n+1})-p^{n+1}\|_{M_{\Gamma}}^2 + \tau\, {\textstyle \sum_{k=0}^n}  \| p(t^{k+1})-p^{k+1}\|_{A_{\Gamma}}^2  \leq C \, \tau^2,\\[2mm]
			\label{th_error_Lie_impl_Euler_b} \tag{b} &	\| u(t^{n+1})-u^{n+1}\|_{A}^2 + \| p(t^{n+1})-p^{n+1}\|_{A_{\Gamma}}^2  \leq C \, \tau,
		\end{align}
	\end{subequations}
	with a constant $C>0$ independent of $\tau$ and $h$.
\end{theorem}
\begin{proof}
	For the sake of brevity, we define the error~$\err{u}^n \coloneqq u(t^n)-u^n$, and analogously $\err{u_1}^n$, $\err{p}^n$. Within this proof, $c$ denotes a positive generic constant, which may change values from line to line, but it is independent of $\tau$ and $h$. Moreover, we use the short notion~$L^1(M)$, $L^2(A_{22})$ (and similar expressions) for the~$L^1([0,T],M)$ and $L^2([0,T],A_{22})$ norm, respectively.	
	
	\medskip\noindent
	\emph{Estimate \ref{th_error_Lie_impl_Euler_a}:} By the interpretation of the Lie splitting scheme~\eqref{eq:LieEuler} as the modified implicit Euler method~\eqref{eqn_modified_Lie_impl_Euler_help}, and using linearity we note that  
	\begin{align}
	&\ \begin{bmatrix}
	M_{11} & M_{12} \\ M_{21} & M_{22}+M_{\Gamma}
	\end{bmatrix}
	\begin{bmatrix}
	\partial_\tau \err{u_1}^{n+1} \\ \partial_\tau \err{p}^{n+1}
	\end{bmatrix}
	+
	\begin{bmatrix}
	A_{11} & A_{12} \\ A_{21} & A_{22} + A_{\Gamma}
	\end{bmatrix}
	\begin{bmatrix}
	\err{u_1}^{n+1} \\ \err{p}^{n+1}
	\end{bmatrix} -
	\begin{bmatrix}
	A_{12}\\ A_{22}
	\end{bmatrix}
	\big(\err{p}^{n+1} - \err{p}^{n}\big) \label{eqn_modified_Lie_impl_Euler} \\*
	= &\   
	\begin{bmatrix}
	M_{11} & M_{12} \\ M_{21} & M_{22} +M_{\Gamma}
	\end{bmatrix}
	\begin{bmatrix}
	\partial_\tau u_1(t^{n+1}) - \dot u_1 (t^{n+1}) \\ 
	\partial_\tau p(t^{n+1}) - \dot p(t^{n+1})
	\end{bmatrix} - 
	\begin{bmatrix}
	M_{12}\\ M_{22}
	\end{bmatrix} (\partial_\tau p^{n+1}- \partial_\tau p^n)
	\notag
	- \begin{bmatrix}
	A_{12}\\ A_{22}
	\end{bmatrix}
	\big(p(t^{n+1}) - p(t^{n})\big).\notag
	\end{align}
	Testing this equation by $\tau\err{u}^{n+1}$, then using the linearity of the equations and similar estimates as in Lemma~\ref{lem_stability_Lie_impl_Euler}, we obtain
	\begin{align}
	&\|\err{u}^{n+1}\|_{M}^2  + \|\err{p}^{n+1}\|_{M_\Gamma}^2 + \sum_{k=0}^n \|\err{p}^{k+1} - \err{p}^{k}\|_{M_\Gamma}^2 + \tau \sum_{k=0}^n \|\err{u}^{k+1}\|_{A}^2 + 2\tau \sum_{k=0}^n \|\err{p}^{k+1}\|_{A_\Gamma}^2 \label{eqn_est_Lie_impl_Euler_help_0} \\
	\leq\, & 2\sum_{k=0}^n \|\err{u}^{k+1}\|_{M} \big( \|u(t^{k+1})-u(t^k) - \tau \dot u (t^{k+1})\|_{M} + \tau \|\partial_\tau p^{k+1}- \partial_\tau p^k\|_{M_{22}}\big) \notag\\
	&+ 2\sum_{k=0}^n \|\err{p}^{k+1}\|_{M_\Gamma} \|p(t^{k+1})-p(t^k) - \tau \dot p (t^{k+1})\|_{M_{\Gamma}} +  \tau \sum_{k=0}^n \| p(t^{k+1})-p(t^k)\|_{A_{22}}^2. \notag
	\end{align}
	Here, we have used that $\err{u}^0 = 0$ and, thus, $\err{p}^0 = 0$. Together with $\sqrt a+ \sqrt b \leq \sqrt 2 \sqrt{a + b}$ and a discrete version of Gronwall's Lemma, e.g., as in~\cite[Lem.~8.13]{Zim21}, it follows that 
	\begin{align}
	&\ \|\err{u}^{n+1}\|_{M}^2  + \|\err{p}^{n+1}\|_{M_\Gamma}^2 
	+ \tau \sum_{k=0}^n \|\err{u}^{k+1}\|_{A}^2 + 2\tau \sum_{k=0}^n \|\err{p}^{k+1}\|_{A_\Gamma}^2  \label{eqn_est_Lie_impl_Euler_help_1}\\
	\leq &\ c \, \Big( \sum_{k=0}^n \tau \|\partial_\tau p^{k+1}-\partial_\tau p^k\|_{M_{22}} \Big)^2
	+ c \, \Big( \sum_{k=0}^n \|u(t^{k+1})-u(t^k) - \tau \dot u (t^{k+1})\|_{M} \Big)^2 \notag\\
	&\ + c \, \Big( \sum_{k=0}^n \|p(t^{k+1})-p(t^k) - \tau \dot p (t^{k+1})\|_{M_{\Gamma}} \Big)^2 
	+ c \, \tau \sum_{k=0}^n \| p(t^{k+1})-p(t^k)\|_{A_{22}}^2 \notag\\
	\leq &\ c \, t^{n+1} \sum_{k=0}^n \tau \|\partial_\tau p^{k+1}-\partial_\tau p^k\|_{M_{22}}^2
	+ c \, \tau^2 \big(\| \ddot{u}\|_{L^1(M)}^2 + \| \ddot{p}\|_{L^1(M_{\Gamma})}^2 + \| \dot{p}\|_{L^2(A_{22})}^2 \big). \notag
	\end{align}
	We study the difference of the discrete derivatives $\partial_\tau p^{k+1}-\partial_\tau p^k$. Using Lemma~\ref{lem_stability_Lie_impl_Euler} for the difference of two consecutive time steps, we have 
	\begin{align}
	&\ \tau \sum_{k=1}^n (1 - 4 c_M h - c_{\alpha,M} \tau h - c_A \tau h^{-1}) \|\partial_\tau p^{k+1}-\partial_\tau p^k\|_{M_{\Gamma}}^2 \label{eqn_est_Lie_impl_Euler_help_2}\\
	\leq\, & \|u^1- u^0\|_{A}^2 + \|p^1 - p^0\|_{A_{\Gamma}}^2 + 2\tau \|\partial_\tau p^1-  \partial_\tau p^0\|_{M_{22}}^2 \notag
	+ \tau \sum_{k=1}^n  \|f_\Omega^{k+1}-f_\Omega^{k}\|_{M^{-1}}^2 + \|f_\Gamma^{k+1}-f_\Gamma^k\|_{M^{-1}_\Gamma}^2 \notag\\
	\leq &\ \frac{1 + c_M h - 2c_{\alpha,M} \tau h  - 2c_A \tau h^{-1}}{1- c_M h - 2c_{\alpha,M} \tau h - 2c_A \tau h^{-1}} \, \Big( \tau^2 \|\dot{u}(0)\|_{A}^2 + \tau^2 \|\dot{p}(0)\|_{A_{\Gamma}}^2 + \tau^2 \|\dot{p}(0)\|_{A_{22}}^2 \notag\\
	&\ \phantom{\frac{1 + c_M h - 2c_{\alpha,M} \tau h  - 2c_A \tau h^{-1}}{1- c_M h - 2c_{\alpha,M} \tau h - 2c_A \tau h^{-1}} \, \Big( }\quad  + \tau \sum_{k=0}^n  \|f_\Omega^{k+1}-f_\Omega^{k}\|_{M^{-1}}^2 +  \|f_\Gamma^{k+1}-f_\Gamma^k\|_{M^{-1}_\Gamma}^2 \Big) , \notag
	\end{align}
	where the  second inequality follows by Lemma~\ref{lem_local_error_Lie_impl_Euler}.
	Combining the estimates~\eqref{eqn_est_Lie_impl_Euler_help_1} and~\eqref{eqn_est_Lie_impl_Euler_help_2}, we finally get 
	\begin{align*}
	&\ \|\err{u}^{n+1}\|_{M}^2  + \|\err{p}^{n+1}\|_{M_\Gamma}^2 + \frac12 \tau \sum_{k=0}^n \|\err{u}^{k+1}\|_{A}^2 + 2\tau \sum_{k=0}^n \|\err{p}^{k+1}\|_{A_\Gamma}^2 \\
	\leq &\ c \, \tau^2\, \frac{(1+ c_M h - 2c_{\alpha,M} \tau h  - 2c_A \tau h^{-1})}{(1 - 4 c_M h - c_{\alpha,M} \tau h - c_A \tau h^{-1})(1- c_M h - 2c_{\alpha,M} \tau h  - 2c_A \tau h^{-1}) } \, t^{n+1}  \\*
	&\ \quad \cdot \Big(c_M h \big[\|\dot{u}(0)\|_{A}^2 + \|\dot{p}(0)\|_{A_{\Gamma}}^2 + \| \dot{f}_\Omega\|_{L^2(M^{-1})}^2 
	+ \| \dot{f}_\Gamma\|_{L^2(M_{\Gamma}^{-1})}^2 \big] + (c_{\alpha,M}h^2+c_A)\|\dot{p}(0)\|_{M_\Gamma}^2 \Big) \\*
	&\ +  c \, \tau^2\,\Big(\| \ddot{u}\|_{L^1(M)}^2 +  \| \ddot{p}\|_{L^1(M_{\Gamma})}^2 +   \| \dot{p}\|_{L^2(A_{22})}^2 \Big) .
	\end{align*}
	
	\medskip\noindent
	\emph{Estimate \ref{th_error_Lie_impl_Euler_b}:}
	For the second estimate, we test equation~\eqref{eqn_modified_Lie_impl_Euler} by $\err{u}^{n+1} - \err{u}^{n}$. Summation from $k=0$ to $n$ yields 
	\begin{align*}
	&\ \tau \sum_{k=0}^n (1 - 2c_{\alpha,M} \tau h - 2\tfrac{c_A \tau} h) \big\| \partial_\tau \err{p}^{k+1} \big\|_{M_\Gamma}^2 + \| \err{u}^{n+1}\|_A^2 + \| \err{p}^{n+1}\|_{A_\Gamma}^2 +\sum_{k=0}^n \| \err{p}^{k+1} - \err{p}^{k} \|_{A_\Gamma}^2\\
	\leq &\  \tau c_M h   \sum_{k=0}^n \big\|\partial_\tau p^{k+1}  - \partial_\tau p^k\big\|_{M_{\Gamma}}^2  + \tau \sum_{k=0}^n \big\|\partial_\tau u(t^{k+1}) - \dot u (t^{k+1})\big\|_{M}^2\\
	&\ + \tau \sum_{k=0}^n \big\|\partial_\tau p(t^{k+1}) - \dot p (t^{k+1})\big\|_{M_{\Gamma}}^2 +  2 \sum_{k=0}^n \| p(t^{k+1})-p(t^k)\|_{A_{22}}^2.
	\end{align*}
	By the same steps as in the proof of estimate~\ref{th_error_Lie_impl_Euler_a}, we finally obtain the second claim. 
\end{proof}

\begin{remark}\label{rem_nonlin}
	Theorem~\ref{th_error_Lie_impl_Euler} still holds for state-dependent, locally Lipschitz continuous right-hand sides~$f_\Omega$ and~$f_\Gamma$; see~\ref{app:semi-linear} and, in particular, Theorem~\ref{th_error_Lie_impl_Euler_nonlin}.
\end{remark}
\begin{remark}\label{rem_CFLnotnecessary}
	Although we have assumed the weak CFL condition in the proof of Theorem~\ref{th_error_Lie_impl_Euler}, the numerical experiments of Section~\ref{sec:numerics} indicate that this condition may not be necessary to obtain first-order convergence. 
\end{remark}
%
%
\section{Numerical experiments}\label{sec:numerics}
This section is devoted to convergence experiments for the Lie splitting method proposed in this paper. One aim is to illustrate the convergence results of Theorem~\ref{th_error_Lie_impl_Euler} and to analyze the necessity of the weak CFL condition, cf.~Remark~\ref{rem_CFLnotnecessary}. Moreover, we present as an outlook first experiments of corresponding Strang splitting schemes. 

The presented convergence experiments report on errors between the numerical solution and the (interpolation of the) known exact solution for varying time step size $\tau$. The errors in the bulk and on the surface components are measured in the Bochner norms  
\[
L^\infty(L^2(\Omega))
\coloneqq L^\infty([0,T],L^2(\Omega)), \qquad
L^\infty(L^2(\Gamma))
\coloneqq L^\infty([0,T],L^2(\Gamma)). 
\] 
All the experiments are carried out on the unit disk, i.e., $\Omega=\{x\in \R^2\ |\ \|x\|< 1\}$, using a sequence of meshes obtained by \textsc{DistMesh}; cf.~\cite{distmesh}. The number of degrees of freedom are given by~$159, \dotsc, 5161$, which correspond to the mesh sizes~$h_k \approx \sqrt{2} \, h_{k-1}$. 

Within this section, we focus on the error caused by the splitting approach together with the temporal discretization. The resulting convergence plots often show two regions: one where the temporal error dominates, indicating the convergence order, and one where the error curves flatten out due to the dominating spatial error. 
%
%
\subsection{Lie splitting}\label{sec:numerics:Lie}
We compare the theoretical result of Theorem~\ref{th_error_Lie_impl_Euler} with numerical experiments. We first consider a linear problem before we turn to the nonlinear Allen--Cahn equation. 
%
%
\subsubsection{Failure of the naive PDAE approach}\label{sec:numerics:Lie:naive}
Consider~\eqref{eq:dynamicBC} with constants~$\alpha_\Omega=\alpha_\Gamma=0$, $\beta=\kappa =1$, (linear) right-hand sides, and an initial condition such that the exact solution reads $u(t,x,y) = \exp(-t)\cos(10t)\, x y$. 
%
%
The naive splitting approach discussed in Section~\ref{sec:Lie:naive} with an implicit Euler discretization would yield the scheme 
\begin{align*}
M_{11} \partial_\tau u_1^{n+1} + A_{11} u_1^{n+1} 
= &\  f_{1}^{n+1} - A_{12} p^n, \\
M_\Gamma \partial_\tau p^{n+1} + A_\Gamma p^{n+1} 
= &\ f_\Gamma^{n+1} + f_{2}^{n+1} - A_{22} p^n - A_{21} u_1^{n+1}.
\end{align*}
Note that this scheme does \emph{not} include any information on the derivatives of~$u_1$ or~$p$. Similarly, also the (Lie versions of the) force and component splitting approaches discussed in Section~\ref{section:splittings - motivation} do not include such information. In this regard, these three splitting schemes are \emph{naive}. 

It is illustrated in Figure~\ref{fig:LieNaive} that these approaches approximate a different system. To see this, we have computed the 'optimal' error of a time stepping scheme. This bound corresponds to the spatial error of the discretization and is indicated by the gray dashed line in the plots. One can clearly observe that the schemes without information on the derivatives do not reach this bound, whereas the Lie splitting introduced in Section~\ref{sec:Lie:implEuler} does reach the optimum.
\begin{figure}
	\centering
%
%
\begin{tikzpicture}

\begin{axis}[%
width=2.3in,
height=2.0in,
at={(-2.7in,0.in)},
scale only axis,
xmode=log,
xmin=1.2207e-05,
xmax=0.025,
xminorticks=true,
xlabel={step size $\tau$},
ymode=log,
ymin=0.0008,
ymax=0.12,
yminorticks=true,
axis background/.style={fill=white},
legend columns = 4,
legend style={legend cell align=left, align=left, at={(1.089,1.05)}, anchor=south, draw=white!15!black}
]
\addplot [color=mycolor1, line width=1.3, mark=triangle*]
  table[row sep=crcr]{%
0.025	0.022249092\\
0.0175438596491228	0.014480457\\
0.0125	0.010534404\\
0.00884955752212389	0.0095915896\\
0.00625	0.010484536\\
0.00442477876106195	0.011618851\\
0.003125	0.012582261\\
0.0022075055187638	0.013319012\\
0.0015625	0.013861019\\
0.00078125	0.01454116\\
0.000390625	0.014890114\\
0.0001953125	0.015066707\\
9.765625e-05	0.015155504\\
4.8828125e-05	0.015200025\\
2.4414e-05    0.015222\\
1.2207e-05    0.015233\\
};
\addlegendentry{force splitting}

\addplot [color=mycolor2, mark=*, line width=1.3]
  table[row sep=crcr]{%
0.025	0.047929023\\
0.0208333333333333	0.037985679\\
0.0175438596491228	0.030037092\\
0.0149253731343284	0.02360442\\
0.0125	0.017580901\\
0.0105263157894737	0.012670622\\
0.00884955752212389	0.0085968019\\
0.00740740740740741	0.0058945749\\
0.00625	0.0045913743\\
0.00526315789473684	0.0039226209\\
0.00442477876106195	0.0042983991\\
0.00371747211895911	0.0057248252\\
0.003125	0.007130072\\
0.0022075055187638	0.0094422959\\
0.0015625	0.011113955\\
0.00078125	0.013168035\\
0.000390625	0.014203874\\
0.0001953125	0.014723675\\
9.765625e-05	0.014984011\\
4.8828125e-05	0.015114285\\
2.4414e-05    0.015179\\
1.2207e-05    0.015212\\
};
\addlegendentry{component splitting}

\addplot [color=mycolor3, line width=1.3, mark=square*]
  table[row sep=crcr]{%
0.025	0.026666311\\
0.0175438596491228	0.021997185\\
0.0125	0.020916879\\
0.00884955752212389	0.021719367\\
0.00625	0.02307851\\
0.00442477876106195	0.024326976\\
0.003125	0.025327829\\
0.0022075055187638	0.026080047\\
0.0015625	0.026628792\\
0.00078125	0.027313405\\
0.000390625	0.027663174\\
0.0001953125	0.027839829\\
9.765625e-05	0.027928597\\
4.8828125e-05	0.027973082\\
2.4414e-05    0.027995\\
1.2207e-05    0.028006\\
};
\addlegendentry{naive PDAE splitting}

\addplot [color=mycolor4, line width=1.3, mark=x]
  table[row sep=crcr]{%
0.025	0.028699239\\
0.0175438596491228	0.020113697\\
0.0125	0.01427267\\
0.00884955752212389	0.010003272\\
0.00625	0.0069430931\\
0.00526315789473684	0.0057798295\\
0.00371747211895911	0.0039763726\\
0.003125	0.0033124763\\
0.0022075055187638	0.0026139963\\
0.0015625	0.0021495029\\
0.00078125	0.0016218908\\
0.000390625	0.0013896847\\
0.0001953125	0.0012914743\\
9.765625e-05	0.0012497294\\
4.8828125e-05	0.0012312408\\
2.44140625e-05	0.0012226722\\
1.220703125e-05	0.0012185674\\
};
\addlegendentry{Lie splitting}

\addplot [color=gray, dashed, line width=1.0pt, forget plot]
  table[row sep=crcr]{%
0.025	0.0012145862\\
1.2207e-05 	0.0012145862\\
};
\end{axis}
\begin{axis}[%
width=2.3in,
height=2.0in,
scale only axis,
xmode=log,
xmin=1.2207e-05,
xmax=0.025,
xminorticks=true,
xlabel={step size $\tau$},
ymode=log,
ymin=0.00215,
ymax=0.3,
yminorticks=true,
yticklabel pos=right,
axis background/.style={fill=white},
legend style={legend cell align=left, align=left, at={(0.97,0.03)}, anchor=south east, draw=white!15!black}
]
\addplot [color=mycolor1, line width=1.3, mark=triangle*]
  table[row sep=crcr]{%
0.025	0.079466736\\
0.0175438596491228	0.06237251\\
0.0125	0.050699185\\
0.00884955752212389	0.04222924\\
0.00625	0.036287931\\
0.00442477876106195	0.032261053\\
0.003125	0.029518751\\
0.0022075055187638	0.027680828\\
0.0015625	0.026453973\\
0.00078125	0.025060116\\
0.000390625	0.024408332\\
0.0001953125	0.024095253\\
9.765625e-05	0.023942133\\
4.8828125e-05	0.023866452\\
2.4414e-05    0.023829\\
1.2207e-05     0.02381\\
};

\addplot [color=mycolor2, line width=1.3, mark=*]
  table[row sep=crcr]{%
0.025	0.098781788\\
0.0208333333333333	0.081471683\\
0.0175438596491228	0.067298107\\
0.0149253731343284	0.055930199\\
0.0125	0.045350439\\
0.0105263157894737	0.03679631\\
0.00884955752212389	0.029701435\\
0.00740740740740741	0.023812293\\
0.00625	0.019400068\\
0.00526315789473684	0.016026209\\
0.00371747211895911	0.016258569\\
0.0026246719160105	0.017298241\\
0.0015625	0.019248589\\
0.00078125	0.021296148\\
0.000390625	0.022493577\\
0.0001953125	0.023130678\\
9.765625e-05	0.023458172\\
4.8828125e-05	0.023624068\\
2.4414e-05    0.023708\\
1.2207e-05    0.023749\\
};

\addplot [color=mycolor3, line width=1.3, mark=square*]
  table[row sep=crcr]{%
0.025	0.12836682\\
0.0175438596491228	0.10512629\\
0.0125	0.089568458\\
0.00884955752212389	0.078861673\\
0.00625	0.071636666\\
0.00442477876106195	0.066935381\\
0.003125	0.063831968\\
0.0022075055187638	0.061803342\\
0.0015625	0.060470568\\
0.00078125	0.058974823\\
0.000390625	0.05827942\\
0.0001953125	0.057945574\\
9.765625e-05	0.057782253\\
4.8828125e-05	0.057701506\\
2.4414e-05    0.057661\\
1.2207e-05    0.057641\\
};

\addplot [color=mycolor4, line width=1.3, mark=x]
  table[row sep=crcr]{%
0.025	0.093213899\\
0.0175438596491228	0.067194626\\
0.0125	0.04922141\\
0.00884955752212389	0.035976536\\
0.00625	0.026434075\\
0.00442477876106195	0.01966961\\
0.003125	0.014823474\\
0.0022075055187638	0.011386958\\
0.0015625	0.0089636133\\
0.00078125	0.006021435\\
0.000390625	0.0045487411\\
0.0001953125	0.0038308774\\
9.765625e-05	0.0034816762\\
4.8828125e-05	0.0033072418\\
2.44140625e-05	0.0032200756\\
1.220703125e-05	0.0031765068\\
};

\addplot [color=gray, dashed, line width=1.0pt, forget plot]
  table[row sep=crcr]{%
0.025	0.003132948\\
1.2207e-05	0.003132948\\
};

\end{axis}
\end{tikzpicture}%
	\caption{Temporal convergence test for different Lie splittings as described in Section~\ref{sec:numerics:Lie:naive} for the mesh size~$h=0.6038$. Plots show the~$L^\infty(L^2(\Om))$-error in $u$ (left) and the~$L^\infty(L^2(\Ga))$-error in $p$ (right). The gray dashed lines indicate the spatial error. }
	\label{fig:LieNaive}
\end{figure}
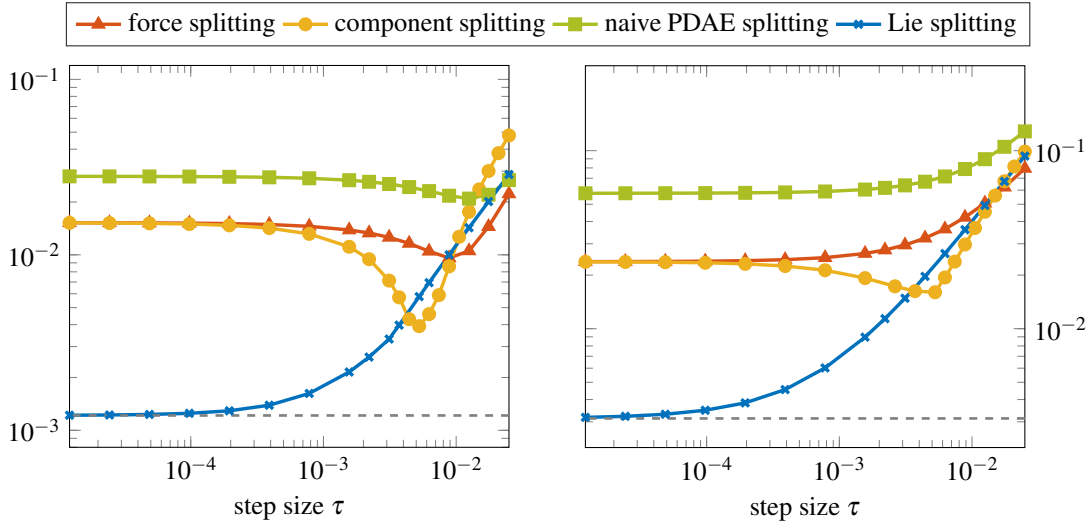
%
%
\subsubsection{Necessity of the weak CFL condition}\label{sec:numerics:Lie:CFL}
In this second experiment, we study the first-order convergence predicted in Theorem~\ref{th_error_Lie_impl_Euler} and analyze the necessity of Assumption~\ref{ass:CFL}. For this, we consider the same parameters as in the previous experiment but with exact solution~$u(t,x,y) = \exp(-t)\, x y$. 

Figure~\ref{fig:weakCFL} shows that we indeed have first-order convergence if a mild CFL condition is satisfied. Within the plot, we consider different mesh sizes~$h$ and set the time step size to~$\tau \coloneqq 1/\lceil 1/h \rceil$ such that $\tau^{-1}\in\N$ and~$\tau\le h$. Note that even the $H^1$-error in $p$ converges with order 1, whereas $u$ only converges with order $1/2$ in the stronger norm. 
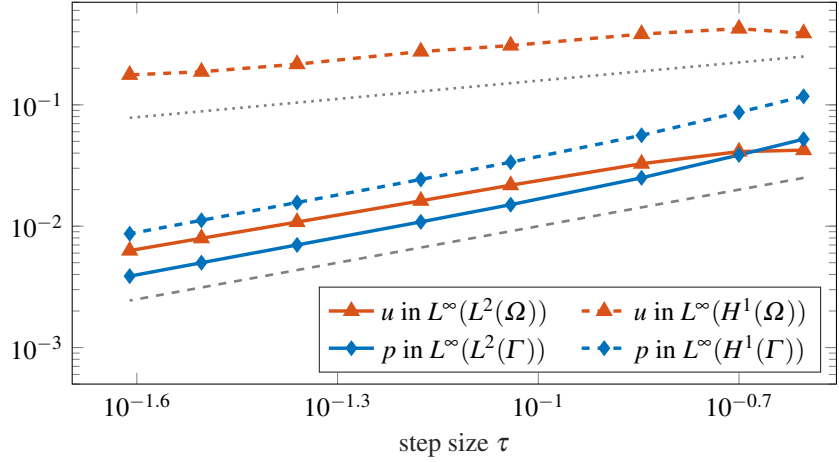
\begin{figure}
	\centering
%
%
\begin{tikzpicture}
\begin{axis}[%
width=4.0in,
height=2.0in,
scale only axis,
xmode=log,
xmin=0.02,
xmax=0.28,
xminorticks=true,
xtick={0.2, 0.1, 0.05, 0.025},
xlabel style={font=\color{white!15!black}},
xlabel={step size~$\tau$},
ymode=log,
ymin=5e-04,
ymax=7e-01,
yminorticks=true,
ylabel style={font=\color{white!15!black}},
axis background/.style={fill=white},
title style={font=\bfseries},
legend columns = 2,
legend style={at={(0.98,0.03)}, anchor=south east, legend cell align=left, align=left, draw=white!15!black}
]
\addplot [color=mycolor1, line width=1.3, mark=triangle*, mark size=2.5]
  table[row sep=crcr]{%
0.25	0.042305002\\
0.2	0.041156321\\
0.142857142857143	0.032709166\\
0.0909090909090909	0.021758359\\
0.0666666666666667	0.016186741\\
0.0434782608695652	0.010813785\\
0.03125	0.0079462385\\
0.024390243902439	0.0062920106\\
};
\addlegendentry{$u$ in $L^\infty(L^2(\Omega))$\quad}

\addplot [color=mycolor1, line width=1.3, dashed, mark=triangle*, mark size=2.5, mark options={solid}]
table[row sep=crcr]{%
	0.25	0.38845317229506\\
	0.2	0.424401002608115\\
	0.142857142857143	0.382641355090891\\
	0.0909090909090909	0.307642746125191\\
	0.0666666666666667	0.274662543038592\\
	0.0434782608695652	0.216200508001337\\
	0.03125	0.187262789803189\\
	0.024390243902439	0.17648900391573\\
};
\addlegendentry{$u$ in $L^\infty(H^1(\Omega))$}

\addplot [color=mycolor4, line width=1.3, mark=diamond*]
table[row sep=crcr]{%
	0.25	0.05210143\\
	0.2	0.038590312\\
	0.142857142857143	0.025026299\\
	0.0909090909090909	0.015038757\\
	0.0666666666666667	0.010846519\\
	0.0434782608695652	0.0070045706\\
	0.03125	0.0049970793\\
	0.024390243902439	0.003875148\\
};
\addlegendentry{$p$ in $L^\infty(L^2(\Gamma))$\quad }

\addplot [color=mycolor4, line width=1.3, dashed, mark=diamond*, mark options={solid}]
table[row sep=crcr]{%
	0.25	0.117450322175947\\
	0.2	0.0866895132628246\\
	0.142857142857143	0.0560733267851749\\
	0.0909090909090909	0.0336642878837115\\
	0.0666666666666667	0.0242659942937494\\
	0.0434782608695652	0.0156666359115319\\
	0.03125	0.0111752245466418\\
	0.024390243902439	0.0086656391507557\\
};
\addlegendentry{$p$ in $L^\infty(H^1(\Gamma))$}


\addplot [color=gray, dashed, line width=1.0pt]
  table[row sep=crcr]{%
0.25	0.025\\
0.2	0.02\\
0.142857142857143	0.0142857142857143\\
0.0909090909090909	0.00909090909090909\\
0.0666666666666667	0.00666666666666667\\
0.0434782608695652	0.00434782608695652\\
0.03125	0.003125\\
0.024390243902439	0.0024390243902439\\
};

\addplot [color=gray, dotted, line width=1.0pt]
  table[row sep=crcr]{%
0.25	0.25\\
0.2	0.2236067977\\
0.142857142857143	0.1889822365\\
0.0909090909090909	0.1507556723\\
0.0666666666666667	0.1290994449\\
0.0434782608695652	0.104257207\\
0.03125	0.08838834765\\
0.024390243902439	0.07808688094\\
};

\end{axis}

\end{tikzpicture}%
	\caption{Temporal convergence test for Lie splitting as described in Section~\ref{sec:numerics:Lie:CFL} with \emph{mild} step size condition $\tau\le h$ for different mesh sizes. Plots show errors in $u$ and $p$. The gray reference lines indicate order~$0.5$ (dotted) and order~$1$ (dashed).}
	\label{fig:weakCFL}
\end{figure}

In Figure~\ref{fig:weakCFL:b} we report on convergence experiments for Lie splitting for different mesh sizes~$h$. Again, we detect first-order convergence. Moreover, there is no $h$-dependence of the convergence, which indicates that the weak CFL condition may not be necessary in practice. 
\begin{figure}
	\centering
%
%
\begin{tikzpicture}

\begin{axis}[%
width=2.3in,
height=2.0in,
at={(-2.7in,0.in)},
scale only axis,
xmode=log,
xmin=0.000651041666666667,
xmax=0.06,
xminorticks=true,
xlabel style={font=\color{white!15!black}},
xlabel={step size $\tau$},
ymode=log,
ymin=1e-04, 
ymax=2e-02, 
yminorticks=true,
ylabel style={font=\color{white!15!black}},
axis background/.style={fill=white},
title style={font=\bfseries},
legend columns = 3,
legend style={legend cell align=left, align=left, at={(1.05,1.05)}, anchor=south, draw=white!15!black}
]
\addplot [color=mycolor0, line width=1.0pt, mark=square]
  table[row sep=crcr]{%
0.05	0.0092849639\\
0.025	0.0044469908\\
0.0125	0.0020982873\\
0.00625	0.0020558742\\
0.003125	0.0022914392\\
0.0015625	0.0024296512\\
0.00078125	0.0025034499\\
};
\addlegendentry{$h=0.20741$\quad}

\addplot [color=mycolor1, line width=1.3, mark=triangle*]
  table[row sep=crcr]{%
0.05	0.010579847\\
0.025	0.0051732995\\
0.0125	0.0025022581\\
0.00625	0.001186285\\
0.003125	0.00081441461\\
0.0015625	0.00095627931\\
0.00078125	0.0010399343\\
};
\addlegendentry{$h=0.14394$\quad}

\addplot [color=mycolor2, line width=1.3pt, mark=x]
  table[row sep=crcr]{%
0.05	0.011445333\\
0.025	0.0056211124\\
0.0125	0.0027380128\\
0.00625	0.0012920353\\
0.003125	0.00058549791\\
0.0015625	0.00042929109\\
0.00078125	0.00052229331\\
};
\addlegendentry{$h=0.093568$}

\addplot [color=mycolor3, mark=*, line width=1.3]
  table[row sep=crcr]{%
0.05	0.01204361\\
0.025	0.0059745771\\
0.0125	0.0029554508\\
0.00625	0.0014385271\\
0.003125	0.00068257769\\
0.0015625	0.00031304961\\
0.00078125	0.00018615612\\
};
\addlegendentry{$h=0.067169$\quad}

\addplot [color=mycolor4, line width=1.3pt, mark=diamond*]
  table[row sep=crcr]{%
0.05	0.012450679\\
0.025	0.0062001348\\
0.0125	0.0030912906\\
0.00625	0.0015269115\\
0.003125	0.0007429051\\
0.0015625	0.00035331376\\
0.00078125	0.00016287017\\
};
\addlegendentry{$h=0.045276$\quad}

\addplot [color=mycolor5, line width=1.3, mark=square*]
  table[row sep=crcr]{%
0.05	0.012724501\\
0.025	0.0063456366\\
0.0125	0.0031758863\\
0.00625	0.0015796241\\
0.003125	0.00077888274\\
0.0015625	0.00037855767\\
0.00078125	0.00017987265\\
};
\addlegendentry{$h=0.032228$}

\addplot [color=gray, dashed, line width=1.0pt, forget plot]
table[row sep=crcr]{%
	0.05	0.005\\
	0.00078125	0.000078125\\
};
\end{axis}


\begin{axis}[%
width=2.3in,
height=2.0in,
scale only axis,
xmode=log,
xmin=0.000651041666666667,
xmax=0.06,
xminorticks=true,
xlabel style={font=\color{white!15!black}},
xlabel={step size $\tau$},
ymode=log,
ymin=1e-04, 
ymax=2e-02, 
yticklabel pos=right,
yminorticks=true,
axis background/.style={fill=white},
title style={font=\bfseries},
legend style={at={(0.97,0.03)}, anchor=south east, legend cell align=left, align=left, draw=white!15!black}
]
\addplot [color=mycolor0, line width=1.0pt, mark=square]
  table[row sep=crcr]{%
0.05	0.012988592\\
0.025	0.0096317172\\
0.0125	0.0079808084\\
0.00625	0.0071618677\\
0.003125	0.0067545131\\
0.0015625	0.006551342\\
0.00078125	0.006449895\\
};

\addplot [color=mycolor1, line width=1.3, mark=triangle*]
  table[row sep=crcr]{%
0.05	0.009905144\\
0.025	0.0063344095\\
0.0125	0.0045858089\\
0.00625	0.0037239727\\
0.003125	0.0032971285\\
0.0015625	0.0030849483\\
0.00078125	0.0029792015\\
};

\addplot [color=mycolor2, line width=1.3pt, mark=x]
  table[row sep=crcr]{%
0.05	0.0088276413\\
0.025	0.0051047939\\
0.0125	0.0032866135\\
0.00625	0.0023936126\\
0.003125	0.0019532925\\
0.0015625	0.0017352664\\
0.00078125	0.0016269062\\
};

\addplot [color=mycolor3, mark=*, line width=1.3]
  table[row sep=crcr]{%
0.05	0.0082376541\\
0.025	0.0044186187\\
0.0125	0.0025410056\\
0.00625	0.0016177346\\
0.003125	0.0011638528\\
0.0015625	0.00094026124\\
0.00078125	0.00082966714\\
};

\addplot [color=mycolor4, line width=1.3pt, mark=diamond*]
  table[row sep=crcr]{%
0.05	0.0080339308\\
0.025	0.0041342492\\
0.0125	0.0022175881\\
0.00625	0.0012706507\\
0.003125	0.00080414196\\
0.0015625	0.00057479933\\
0.00078125	0.00046190351\\
};

\addplot [color=mycolor5, line width=1.3, mark=square*]
  table[row sep=crcr]{%
0.05	0.0079593097\\
0.025	0.0040195944\\
0.0125	0.0020774359\\
0.00625	0.0011150443\\
0.003125	0.0006386235\\
0.0015625	0.00040378837\\
0.00078125	0.00028841214\\
};


\addplot [color=gray, dashed, line width=1.0pt, forget plot]
table[row sep=crcr]{%
	0.05	0.005\\
	0.00078125	0.000078125\\
};
\end{axis}
\end{tikzpicture}%
	\caption{Temporal convergence test for the proposed Lie splitting from Section~\ref{sec:Lie:implEuler} for different mesh sizes~$h$. Plots show $L^\infty(L^2(\Omega))$-error in $u$ (left) and $L^\infty(L^2(\Gamma))$-error in $p$ (right). The gray dashed reference line indicates order~$1$. }
	\label{fig:weakCFL:b}
\end{figure}
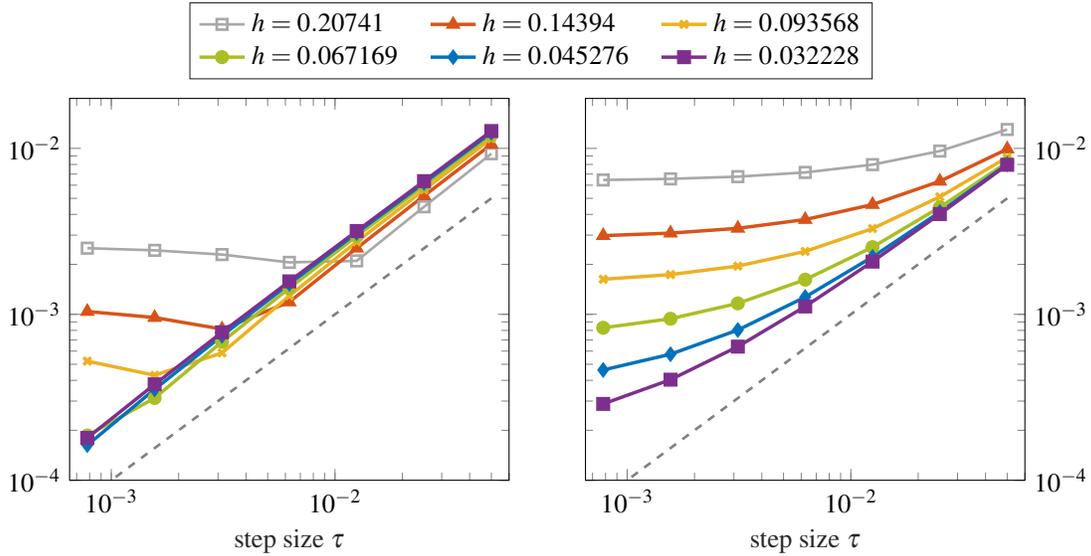
%
%
\subsubsection{Allen--Cahn equation}\label{sec:numerics:Lie:AllenCahn}
Finally, we consider an example with a nonlinearity on the boundary. More precisely, we consider the heat equation in the bulk with the Allen--Cahn-type dynamic boundary condition with a double-well potential, i.e., 
\begin{alignat*}{2}
\dot  u - \Delta u  &= f_\Omega &&\qquad\text{in } \Omega, \\
\dot u - \Delta_\Gamma u +  \partial_\nu u &= f_\Gamma -u^3+u &&\qquad \text{on } \Gamma.  
\end{alignat*}
As the exact solution we took $u(t,x,y)=(x^2+y^2)^2 \cos(\pi t/2)$ and  $f_\Omega$ as well as $f_\Gamma$ are chosen state-independently such that they fit the left-hand side. 
The nonlinear systems, which arise after the proposed discretization, are solved with the help of Newton's method (with starting value $p^n$). The convergence history is illustrated in Figure~\ref{figure:LieEuler:Allan_Cahn}. It shows that the nonlinear term~$-u^3+u$  does not effect the convergence rate as mentioned in Remark~\ref{rem_nonlin}. 
\begin{figure}
	\centering
%
%
\begin{tikzpicture}

\begin{axis}[%
width=2.3in,
height=2.0in,
at={(-2.7in,0.in)},
scale only axis,
xmode=log,
xmin=0.000651041666666667,
xmax=0.06,
xminorticks=true,
xlabel style={font=\color{white!15!black}},
xlabel={step size $\tau$},
ymode=log,
ymin=3e-04, 
ymax=3e-01,	
yminorticks=true,
ylabel style={font=\color{white!15!black}},
axis background/.style={fill=white},
title style={font=\bfseries},
legend columns = 3,
legend style={legend cell align=left, align=left, at={(1.05,1.05)}, anchor=south, draw=white!15!black}
]
\addplot [color=mycolor0, line width=1.0pt, mark=square]
  table[row sep=crcr]{%
0.05	0.03966606\\
0.025	0.021790493\\
0.0125	0.020202323\\
0.00625	0.022162655\\
0.003125	0.023410963\\
0.0015625	0.02411513\\
0.00078125	0.024488225\\
};
\addlegendentry{$h=0.20741$}

\addplot [color=mycolor1, line width=1.3, mark=triangle*]
  table[row sep=crcr]{%
0.05	0.037593027\\
0.025	0.017049075\\
0.0125	0.0092938849\\
0.00625	0.0099675115\\
0.003125	0.011036574\\
0.0015625	0.011733398\\
0.00078125	0.012129551\\
};
\addlegendentry{$h=0.14394$\quad}

\addplot [color=mycolor2, line width=1.3pt, mark=x]
  table[row sep=crcr]{%
0.05	0.040386804\\
0.025	0.018359395\\
0.0125	0.0077824512\\
0.00625	0.0037206643\\
0.003125	0.0042313587\\
0.0015625	0.0046822959\\
0.00078125	0.0049925241\\
};
\addlegendentry{$h=0.093568$}

\addplot [color=mycolor3, mark=*, line width=1.3]
  table[row sep=crcr]{%
0.05	0.042004927\\
0.025	0.019947548\\
0.0125	0.0089706528\\
0.00625	0.0036079157\\
0.003125	0.0018311456\\
0.0015625	0.0020833895\\
0.00078125	0.0023127643\\
};
\addlegendentry{$h=0.067169$\quad}

\addplot [color=mycolor4, line width=1.3pt, mark=diamond*]
  table[row sep=crcr]{%
0.05	0.042956663\\
0.025	0.020924912\\
0.0125	0.0099147289\\
0.00625	0.0044255549\\
0.003125	0.0017257277\\
0.0015625	0.00090742421\\
0.00078125	0.0010333157\\
};
\addlegendentry{$h=0.045276$}

\addplot [color=mycolor5, line width=1.3, mark=square*]
  table[row sep=crcr]{%
0.05	0.043550328\\
0.025	0.021486114\\
0.0125	0.010473063\\
0.00625	0.0049678682\\
0.003125	0.0022203428\\
0.0015625	0.00086230133\\
0.00078125	0.00044093026\\
};
\addlegendentry{$h=0.032228$}

\addplot [color=gray, dashed, line width=1.0pt, forget plot]
table[row sep=crcr]{%
	0.05	0.0125 \\
	0.00078125	0.0001953125\\
};
\end{axis}


\begin{axis}[%
width=2.3in,
height=2.0in,
scale only axis,
xmode=log,
xmin=0.000651041666666667,
xmax=0.06,
xminorticks=true,
xlabel style={font=\color{white!15!black}},
xlabel={step size $\tau$},
ymode=log,
ymin=3e-04, 
ymax=3e-01,	
yticklabel pos=right,
yminorticks=true,
axis background/.style={fill=white},
title style={font=\bfseries},
legend style={at={(0.97,0.03)}, anchor=south east, legend cell align=left, align=left, draw=white!15!black}
]
\addplot [color=mycolor0, line width=1.0pt, mark=square]
  table[row sep=crcr]{%
0.05	0.18110422\\
0.025	0.10157878\\
0.0125	0.063149506\\
0.00625	0.044255695\\
0.003125	0.034887571\\
0.0015625	0.030223104\\
0.00078125	0.027895774\\
};

\addplot [color=mycolor1, line width=1.3, mark=triangle*]
  table[row sep=crcr]{%
0.05	0.1688278\\
0.025	0.089460449\\
0.0125	0.051109985\\
0.00625	0.03225508\\
0.003125	0.022906293\\
0.0015625	0.01825154\\
0.00078125	0.015929124\\
};

\addplot [color=mycolor2, line width=1.3pt, mark=x]
  table[row sep=crcr]{%
0.05	0.16033678\\
0.025	0.081087619\\
0.0125	0.042796087\\
0.00625	0.023970106\\
0.003125	0.014635428\\
0.0015625	0.0099875771\\
0.00078125	0.0076685874\\
};

\addplot [color=mycolor3, mark=*, line width=1.3]
  table[row sep=crcr]{%
0.05	0.15762921\\
0.025	0.078415025\\
0.0125	0.040140861\\
0.00625	0.021323367\\
0.003125	0.01199276\\
0.0015625	0.0073467896\\
0.00078125	0.0050286331\\
};

\addplot [color=mycolor4, line width=1.3pt, mark=diamond*]
  table[row sep=crcr]{%
0.05	0.15625309\\
0.025	0.077056968\\
0.0125	0.038791802\\
0.00625	0.019978752\\
0.003125	0.010650332\\
0.0015625	0.0060054268\\
0.00078125	0.0036877765\\
};

\addplot [color=mycolor5, line width=1.3, mark=square*]
  table[row sep=crcr]{%
0.05	0.1555095\\
0.025	0.076323365\\
0.0125	0.038063186\\
0.00625	0.019252606\\
0.003125	0.0099254105\\
0.0015625	0.0052811123\\
0.00078125	0.0029637616\\
};


\addplot [color=gray, dashed, line width=1.0pt, forget plot]
table[row sep=crcr]{%
	0.05	0.05\\
	0.00078125	0.00078125\\
};
\end{axis}
\end{tikzpicture}%
	\caption{Temporal convergence test for Lie splitting \eqref{eq:LieEuler} for a heat equation / Allen--Cahn coupling. Plots show $L^\infty(L^2(\Omega))$-error in $u$ (left) and $L^\infty(L^2(\Gamma))$-error in $p$ (right). The gray dashed reference line indicates order~$1$. }
	\label{figure:LieEuler:Allan_Cahn}
\end{figure}
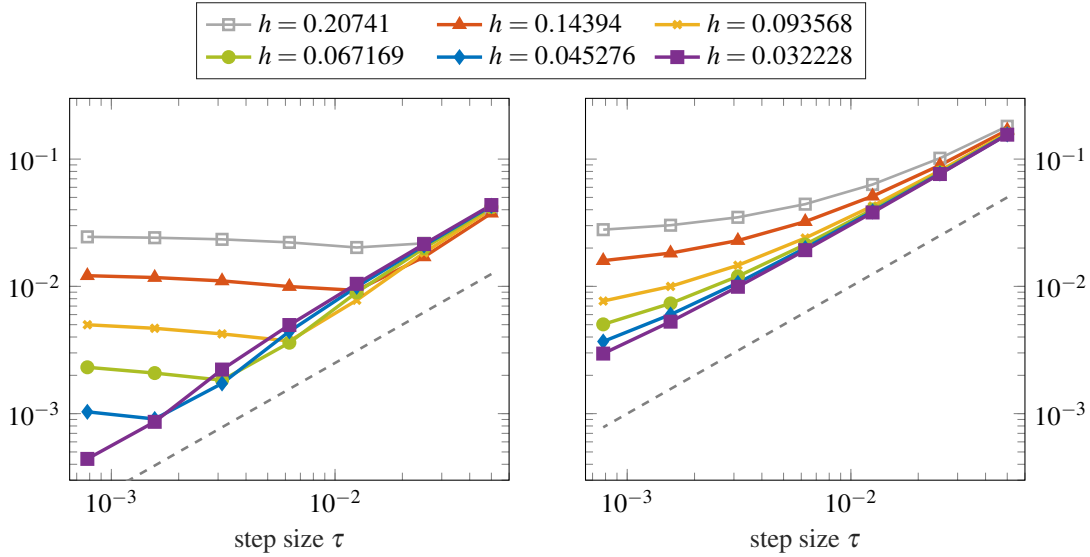
%
%
\subsection{Strang splitting}\label{sec:numerics:Strang}
In this final subsection, we present a numerical experiment for a bulk--surface Strang splitting scheme. For this, we consider the same two subsystems as introduced in Section~\ref{sec:Lie:cont} but in a symmetrized manner. More precisely, we first consider half a time step of the bulk system~\eqref{eq:Lie:a}, then a full time step of the boundary system~\eqref{eq:Lie:b}, and finally the second half step in the bulk. 

To obtain a fully discrete scheme, we need to discuss the temporal discretization of the subsystems. Since we aim for a second-order scheme, a natural choice seems to be the midpoint or trapezoidal rule applied to each of the subsystems. Applied to the linear problem discussed in the beginning of this section, we indeed observe second-order convergence for sufficiently small~$\tau$. However, the convergence is $h$-dependent (similarly as in Figure~\ref{fig:StrangEuler}).

Yet another possibility is to consider 'adjoint' Euler schemes for the first and the last subsystems. More precisely, the {\em explicit} Euler method is applied to the first and the {\em implicit} Euler scheme to the third subsystem. For the second subsystem, we apply the midpoint rule, which is especially suitable, since the first subsystem yields an approximation $u_1^{n+1/2}$. 
Using the notion $\partial_{\tau/2} w^{n+1/2} = 2\tau^{-1} (w^{n+1/2} - w^n)$ on the time interval $[t^n, t^{n+1}]$ of length~$\tau$, we obtain the scheme 
\begin{align*}
M_{11} \partial_{\tau/2} u_1^{n+1/2}  
&= f_{1}^{n} - A_{11} u_1^{n} - M_{12}\partial_\tau p^n - A_{12} p^n, \\
M_\Gamma \partial_\tau p^{n+1} + \tfrac12 A_\Gamma p^{n+1} 
&= f_\Gamma^{n+1/2} + f_{2}^{n+1/2} - \tfrac12 A_\Gamma p^{n}  - M_{22} \partial_\tau p^n - A_{22} p^n - M_{21} \partial_{\tau/2} u_1^{n+1/2} - A_{21} u_1^{n+1/2}, \\
M_{11} \partial_{\tau/2} u_1^{n+1} + A_{11} u_1^{n+1} 
&= f_{1}^{n+1} - M_{12}\partial_\tau p^{n+1} - A_{12} p^{n+1}.
\end{align*}
The corresponding convergence behaviour for the numerical experiment presented in Section~\ref{sec:numerics:Lie:CFL} is shown in Figure~\ref{fig:StrangEuler}. One can observe second-order convergence (for sufficiently small $\tau$), which is $h$-dependent. Moreover, the scheme calls for some kind of CFL condition, which is no surprise due to the inclusion of an explicit scheme. 
\begin{figure}
	\centering
%
%
\begin{tikzpicture}

\begin{axis}[%
width=2.3in,
height=2.0in,
at={(-2.7in,0.in)},
scale only axis,
xmode=log,
xmin=0.000651041666666667,
xmax=0.06,
xminorticks=true,
xlabel style={font=\color{white!15!black}},
xlabel={step size $\tau$},
ymode=log,
ymin=2e-04,
ymax=5e-01,
yminorticks=true,
ylabel style={font=\color{white!15!black}},
axis background/.style={fill=white},
title style={font=\bfseries},
legend columns = 3,
legend style={legend cell align=left, align=left, at={(1.05,1.05)}, anchor=south, draw=white!15!black}
]
\addplot [color=mycolor0, line width=1.0pt, mark=square]
  table[row sep=crcr]{%
0.05	0.019956976\\
0.025	0.0055521046\\
0.0125	0.0028460577\\
0.00625	0.0024052711\\
0.003125	0.0024143839\\
0.0015625	0.0024776232\\
0.00078125	0.0025240043\\
};
\addlegendentry{$h=0.20741$}

\addplot [color=mycolor1, line width=1.3, mark=triangle*]
  table[row sep=crcr]{%
0.05	0.41723058\\
0.025	0.012412125\\
0.0125	0.0031412568\\
0.00625	0.0013117869\\
0.003125	0.0010123079\\
0.0015625	0.0010176689\\
0.00078125	0.0010606489\\
};
\addlegendentry{$h=0.14394$\quad}

\addplot [color=mycolor2, line width=1.3pt, mark=x]
  table[row sep=crcr]{%
0.05	6669486600000\\
0.025	0.074806731\\
0.0125	0.0097682015\\
0.00625	0.002373751\\
0.003125	0.00084222726\\
0.0015625	0.00057085533\\
0.00078125	0.00055811745\\
};
\addlegendentry{$h=0.093568$}

\addplot [color=mycolor3, mark=*, line width=1.3]
  table[row sep=crcr]{%
0.05	3.8836239e+22\\
0.025	1.8524374e+22\\
0.0125	0.03232615\\
0.00625	0.0066280762\\
0.003125	0.0015475443\\
0.0015625	0.00046473417\\
0.00078125	0.00026560332\\
};
\addlegendentry{$h=0.067169$\quad}

\addplot [color=mycolor4, line width=1.3pt, mark=diamond*]
  table[row sep=crcr]{%
0.05	4.1500582e+31\\
0.025	7.204372e+42\\
0.0125	1.9373532e+32\\
0.00625	0.022269688\\
0.003125	0.0047993492\\
0.0015625	0.0010953724\\
0.00078125	0.00028662074\\
};
\addlegendentry{$h=0.045276$}

\addplot [color=mycolor5, line width=1.3, mark=square*]
  table[row sep=crcr]{%
0.05	4.8016008e+40\\
0.025	1.8481645e+62\\
0.0125	3.0922683e+78\\
0.00625	7.7828284e+36\\
0.003125	0.01553388\\
0.0015625	0.0034483645\\
0.00078125	0.00078008227\\
};
\addlegendentry{$h=0.032228$}


\addplot [color=gray, dashed, line width=1.0pt, forget plot]
table[row sep=crcr]{%
0.05	0.25\\
0.025	0.0625\\
0.0125	0.015625\\
0.00625	0.00390625\\
0.003125	9.765625e-04\\
0.0015625	2.44140625e-04\\
0.00078125	6.103515625e-05\\
};

\end{axis}


\begin{axis}[%
width=2.3in,
height=2.0in,
scale only axis,
xmode=log,
xmin=0.000651041666666667,
xmax=0.06,
xminorticks=true,
xlabel style={font=\color{white!15!black}},
xlabel={step size $\tau$},
ymode=log,
ymin=2e-04,
ymax=5e-01,
yticklabel pos=right,
yminorticks=true,
axis background/.style={fill=white},
title style={font=\bfseries},
legend style={at={(0.97,0.03)}, anchor=south east, legend cell align=left, align=left, draw=white!15!black}
]
\addplot [color=mycolor0, line width=1.0pt, mark=square]
  table[row sep=crcr]{%
0.05	0.0582771\\
0.025	0.017892574\\
0.0125	0.0091290539\\
0.00625	0.0070180775\\
0.003125	0.0065044953\\
0.0015625	0.0063818844\\
0.00078125	0.006354054\\
};

\addplot [color=mycolor1, line width=1.3, mark=triangle*]
  table[row sep=crcr]{%
0.05	1.0771895\\
0.025	0.036167421\\
0.0125	0.010384606\\
0.00625	0.0045664758\\
0.003125	0.0032152651\\
0.0015625	0.002918822\\
0.00078125	0.0028648635\\
};

\addplot [color=mycolor2, line width=1.3pt, mark=x]
  table[row sep=crcr]{%
0.05	38247010000000\\
0.025	0.20587096\\
0.0125	0.027103923\\
0.00625	0.007287089\\
0.003125	0.0027706359\\
0.0015625	0.0017431866\\
0.00078125	0.0015310581\\
};

\addplot [color=mycolor3, mark=*, line width=1.3]
  table[row sep=crcr]{%
0.05	2.9972599e+23\\
0.025	1.1626844e+23\\
0.0125	0.083343948\\
0.00625	0.017878721\\
0.003125	0.0045763932\\
0.0015625	0.0015228272\\
0.00078125	0.00084309662\\
};

\addplot [color=mycolor4, line width=1.3pt, mark=diamond*]
  table[row sep=crcr]{%
0.05	3.9196255e+32\\
0.025	6.4221437e+43\\
0.0125	1.198124e+33\\
0.00625	0.056696989\\
0.003125	0.012643699\\
0.0015625	0.0031068372\\
0.00078125	0.00090942632\\
};

\addplot [color=mycolor5, line width=1.3, mark=square*]
  table[row sep=crcr]{%
0.05	5.5006627e+41\\
0.025	2.0827796e+63\\
0.0125	3.1884191e+79\\
0.00625	4.3123044e+37\\
0.003125	0.039106207\\
0.0015625	0.0089021249\\
0.00078125	0.0021285681\\
};


\addplot [color=gray, dashed, line width=1.0pt, forget plot]
table[row sep=crcr]{%
	0.05	0.25\\
	0.025	0.0625\\
	0.0125	0.015625\\
	0.00625	0.00390625\\
	0.003125	9.765625e-04\\
	0.0015625	2.44140625e-04\\
	0.00078125	6.103515625e-05\\
};

\end{axis}
\end{tikzpicture}%
	\caption{Temporal convergence test for Strang splitting with Euler / midpoint discretization for different mesh sizes~$h$. Plots show $L^\infty(L^2(\Omega))$-error in $u$ (left) and $L^\infty(L^2(\Gamma))$-error in $p$ (right). The gray dashed reference line indicates order~$2$. }
	\label{fig:StrangEuler}
\end{figure}
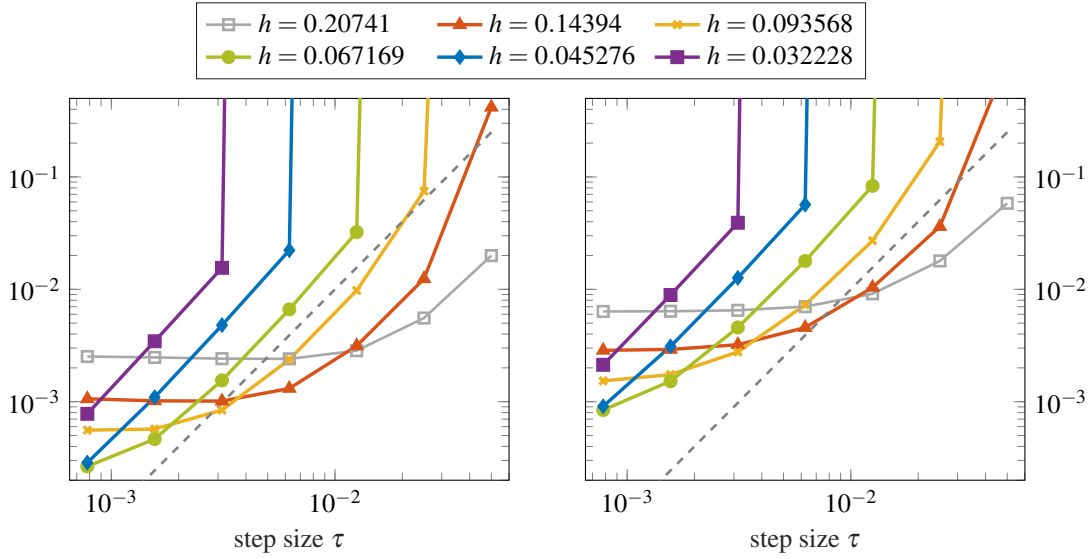
%
%
%
\section{Conclusion}
\label{sec:conclusion}
Within this paper, we have constructed and analyzed a first-order bulk--surface Lie splitting scheme for parabolic problems with dynamic boundary conditions. For this, we have reformulated the system as a coupled system of bulk and surface dynamics. Moreover, it has been observed that such splitting approaches need information on the derivatives of the variables in order to obtain reasonable results.  
The resulting splitting scheme is of particular value in the presence of highly oscillatory or nonlinear boundary conditions. 
Future research will focus on the construction of bulk--surface Strang splitting schemes with the aim of finding a second-order scheme, which is independent of the spatial discretization parameter~$h$. 

\section*{Acknowledgements}
Robert Altmann and Christoph Zimmer acknowledge the support of the Deutsche Forschungsgemeinschaft (DFG, German Research Foundation) through the project 446856041. 

The work of Bal\'azs Kov\'acs is funded by the Heisenberg Programme of the Deutsche Forschungsgemeinschaft (DFG, German Research Foundation) -- Project-ID 446431602.
%
%

\bibliographystyle{IMANUM-BIB}
\bibliography{bib_dynBC}


\appendix
\renewcommand{\thesection}{Appendix~\Alph{section}}
\section{Semi-linear problems}\label{app:semi-linear} 
Within this appendix, we provide details on the extension of the proofs in Section~\ref{sec:Lie:convergence} to the semi-linear case. We show that the convergence order remains one for state-dependent right-hand sides $f_\Omega \colon [0,T] \times \R^\dofOm \to \R^\dofOm$ and $f_\Gamma \colon [0,T] \times \R^\dofGa \to \R^\dofGa$. Although all statements of this section hold for locally Lipschitz continuous right-hand sides, we assume for the sake of brevity that $f_\Omega$ and $f_\Gamma$ are (globally) Lipschitz continuous. To be precise, we suppose that the function $f_\Omega$ satisfies
\begin{equation}
\label{eqn_Lipschitz_fOm}
\| f_\Omega(t;u) - f_\Omega(\tilde{t};\tilde{u})\|_{M^{-1}}^2 
\leq L_\Omega\, |t - \tilde{t}|^2 + L_{A} \|u - \tilde{u}\|_{A}^2 + L_{M} \|u - \tilde{u}\|_{M}^2
\end{equation}
with the constants $L_\Omega,L_A,L_M\geq 0$ and similarly for $f_\Gamma$ with $L_\Gamma,L_{A_\Gamma},L_{M_\Gamma}\geq 0$. 

To prove first-order convergence, we follow the steps of Section~\ref{sec:Lie:convergence}. We start with studying the difference of the discrete derivatives~$\partial_\tau p^{n+1}$ as well as the (local) error caused by the first time step.
\begin{lemma}\label{lem_stability_Lie_impl_Euler_nonlin}
	Let the assumptions of Lemma~\ref{lem_stability_Lie_impl_Euler} be fulfilled and suppose that the right-hand sides $f_\Omega$ and $f_\Gamma$ are Lipschitz continuous. We introduce the non-negative constant~$\eta_\Omega$ by 
	\begin{equation*}
	\eta_\Omega \coloneqq \begin{cases} \sqrt{2L_{M}}, & \text{ if } L_{A} <  \sqrt{2L_{M}},\\
	L_{A} + \sqrt{L_{A}^2 - 2L_{M}}, & \text{ if } L_{A} \geq  \sqrt{2L_{M}},
	\end{cases}
	\end{equation*}
	and analogously~$\eta_\Gamma$ with~$\frac 1 2 L_{A_\Gamma}$ and~$\frac 1 2 L_{M_\Gamma}$ in place of $L_{A}$ and $L_M$, respectively. Let  $L\coloneqq L_{\Omega} + L_{\Gamma}$. If $\tfrac 1 {2\tau} > C\coloneqq \max\{ 2 L_{A}, \sqrt{8L_{M}},L_{A_\Gamma}, 2 \sqrt{L_{M_\Gamma}}\}$ 
	holds, then for the semi-linear system~\eqref{eq:semidiscreteDAE}, every discrete solution given by the Lie splitting~\eqref{eq:LieEuler} satisfies
	\begin{align*}
	& \tau  \sum_{k=1}^n \big(1 - c_M h(4 + 4\tau^2  L_{M}) - (c_{\alpha,M} \tau h + c_A \tau h^{-1})(1+4 \tau^2  L_{A})\big)\big\| \partial_\tau p^{k+1} - \partial_\tau p^k\big\|_{M_{\Gamma}}^2\\*
	\leq\, & e^{2 t^{n}C}\Big[ \|u^1-u^{0}\|_{A+\eta_\Omega M}^2  + \| p^1 - p^{0}\|_{A_{\Gamma}+\eta_\Gamma M_\Gamma}^2 + 2\tau\, \|\partial_\tau p^1 - \partial_\tau p^0\|_{M_{22}}^2 + 2 L \tau^2 \sum_{k=0}^{n-1} e^{-2t^{k}C} \tau \Big].
	\end{align*}
\end{lemma}

\begin{proof}
	We consider~\eqref{eq:LieEuler_matrix} for the difference of two consecutive time steps and add $\eta_{\Omega} M (u^{n+1}-u^n)$ and $[
	0\ \eta_{\Gamma} M_\Gamma]^T (p^{n+1}-p^n)$
	to both sides of the system. Then, by following the steps of Lemma~\ref{lem_stability_Lie_impl_Euler}, we obtain the bound
	%
	\begin{align*}
	&\|u^{n+1} - u^n\|_{A+\eta_{\Omega}M}^2 - \|u^1 - u^0\|_{A+\eta_{\Omega}M}^2 + \| p^{n+1}-p^n\|_{A_{\Gamma}+\eta_{\Gamma}M_\Gamma}^2 - \|p^1 - p^0\|_{A_{\Gamma} +\eta_{\Gamma}M_\Gamma}^2\\
	&\qquad - 2\tau\, \|\partial_\tau p^1 - \partial_\tau p^0\|_{M_{22}}^2 + \tau  \sum_{k=1}^n (1 - 4 c_M h - c_{\alpha,M} \tau h - c_A \tau h^{-1})\, \big\| \partial_\tau p^{k+1} - \partial_\tau p^{k}\big\|_{M_{\Gamma}}^2\\
	\leq\, &  \tau \sum_{k=1}^n \, \|f_\Omega(t^{k+1};u_1^{k+1},p^k) - 
	f_\Omega(t^{k};u_1^{k},p^{k-1}) +\eta_\Omega M (u^{k+1}-u^k)\|_{M^{-1}}^2\\*
	& \qquad \qquad  + \, \big\|f_\Gamma(t^{k+1};p^{k+1}) - f_\Gamma(t^{k};p^{k}) +\eta_\Gamma M_\Gamma (p^{k+1}-p^k) \big\|_{M^{-1}_\Gamma}^2\\
	\leq\, & 2\tau \sum_{k=1}^n  L \tau^2 + L_{A}\, \bigg\|\begin{bmatrix}
	u^{k+1}_1 \\ p^k
	\end{bmatrix}  - \begin{bmatrix}
	u^{k}_1 \\ p^{k-1}
	\end{bmatrix} \bigg\|_{A}^2 +  L_{M}\, \bigg\|\begin{bmatrix}
	u^{k+1}_1 \\ p^k
	\end{bmatrix}  - \begin{bmatrix}
	u^{k}_1 \\ p^{k-1}
	\end{bmatrix} \bigg\|_{M}^2 + \eta_\Omega^2\, \|u^{k+1} - u^{k}\|_{M}^2 \\
	& \qquad \qquad + L_{A_\Gamma}\, \|p^{k+1} - p^{k}\|_{A_\Gamma}^2 + ( L_{M_\Gamma} + \eta_\Gamma^2)\, \|p^{k+1} - p^{k}\|_{M_\Gamma}^2 \\
	\leq\, & 2\tau \sum_{k=1}^n  L \tau^2 + 2 L_{A}  \|u^{k+1} - u^{k}\|_{A}^2 + ( 2L_{M} + \eta_\Omega^2)\, \|u^{k+1} - u^{k}\|_{M}^2 \\ 
	& \qquad  \qquad + L_{A_\Gamma} \|p^{k+1} - p^{k}\|_{A_\Gamma}^2 + (L_{M_\Gamma} + \eta_\Gamma^2)\, \|p^{k+1} - p^{k}\|_{M_\Gamma}^2 \\
	& \qquad \qquad + 2L_{A} \| p^{k+1} -2p^k +p^{k-1} \|_{A_{22}}^2 + 2L_{M} \| p^{k+1} -2p^k +p^{k-1} \|_{M_{22}}^2\\
	\leq \, &  2\tau \sum_{k=1}^n  L \tau^2 + C\, \big( \|u^{k+1}-u^k\|_{A+\eta_{\Omega}M}^2 + \| p^{k+1} - p^k\|_{A_{\Gamma}+\eta_{\Gamma}M_\Gamma}^2\big)\\
	& \qquad \qquad + 2L_{A} \| p^{k+1} -2p^k +p^{k-1} \|_{A_{22}}^2 + 2 L_{M} \| p^{k+1} -2p^k +p^{k-1} \|_{M_{22}}^2. 
	\end{align*}
	Note that we have used the definition of $\eta_\Omega$ and $\eta_\Gamma$ in the last estimate. The statement follows by a discrete version of Gronwall's lemma; see, e.g., \cite[Prop.~3.1]{Emm99}.
\end{proof}

\begin{lemma}\label{lem_local_error_Lie_impl_Euler_nonlin}
	Let $\eta_\Omega$, $\eta_\Gamma$, $C$, and $L$ be defined as in Lemma~\ref{lem_stability_Lie_impl_Euler_nonlin}. Suppose that the assumptions of Lemma~\ref{lem_local_error_Lie_impl_Euler} are satisfied. Then we have
	\begin{align*}
	&(1-2\tau C)\,\| u^1 - u^0\|_{A+\eta_\Omega M}^2 + (1-2\tau C)\,\| p^1 - p^0\|_{A_\Gamma+\eta_\Gamma M_\Gamma}^2\\
	& \quad+ \tau\, \big(1 - c_M h(1+8\tau^2 L_{M}) - (c_{\alpha,M} \tau h  + c_A \tau h^{-1})(2+8 \tau^2 L_{A})\big)\, \| \partial_\tau p^1 - \dot p(0)\|_{M_\Gamma}^2\\ 
	& \quad \leq \tau^2\, \|\dot{u}(0)\|_{A+\eta_\Omega M}^2 + \tau^2\, \|\dot{p}(0)\|_{A_\Gamma+\eta_\Gamma M_\Gamma}^2 + 2 \tau^2\, \|\dot{p}(0)\|_{A_{22}}^2 + 4\tau^3 C\, \|\dot{p}(0)\|_{A_{22}+\eta_\Omega M_{22}}^2
	+ 2\tau^3 L.
	\end{align*}
\end{lemma}
\begin{proof}
	Following the lines of Lemma~\ref{lem_local_error_Lie_impl_Euler} and~\ref{lem_stability_Lie_impl_Euler_nonlin} we have
	\begin{align*}
	&\tau\, \big(1 - c_M h - 2c_{\alpha,M} \tau h  - 2c_A \tau h^{-1}\big) \| \partial_\tau p^1- \dot p(0)\|_{M_\Gamma}^2 + \|u^1-u^0\|_{A + \eta_\Omega M}^2\\*
	& \qquad - \|\tau  \dot u(0)\|_{A +\eta_\Omega M}^2 +\|p^1-p^0\|_{A_\Gamma + \eta_\Gamma M_\Gamma}^2 - \|\tau \dot p(0)\|_{A_\Gamma + \eta_\Gamma M_\Gamma}^2 - 2\, \|\tau \dot p(0)\|_{A_{22}}^2\\
	\leq\, &  \tau\, \|
	f_\Omega(t^{1};u_1^{1},p^0) - f_\Omega(t^{0};u_1^{0},p^{0}) +\eta_\Omega M (u^{1}-u^0)\|_{M^{-1}}^2\\
	& \qquad + \tau\, \|f_\Gamma(t^{1};p^{1}) - f_\Gamma(t^0;p^0) +\eta_\Gamma M_\Gamma (p^{1}-p^0)\|_{M^{-1}_\Gamma}^2\\
	\leq\, &  2 \tau^3 L + 2\tau L_{A}\, \|
	u_1^{1} -u_1^{0}\|_{A_{11}}^2 + 2 \tau L_{M}\, \|
	u_1^{1} -u_1^{0}\|_{M_{11}}^2 \\
	& \qquad  + 2\eta_\Omega^2\, \| u^1-u^0\|_M^2 + 2 \tau C\, \|p^{1} - p^0\|_{A_\Gamma + \eta_\Gamma M_\Gamma}^2\\
	\leq\, &  2 \tau^3 L + 2 \tau C\, \| u^1-u^0\|_{A+\eta_\Omega M}^2  + 2 \tau C\, \|p^{1} - p^0\|_{A_\Gamma + \eta_\Gamma M_\Gamma}^2 
	\\  & \qquad + 8 \tau^3 L_{A} \big( \| \partial_\tau p^1- \dot p(0)\|_{A_{22}}^2 + \| \dot p(0)\|_{A_{22}}^2\big) + 8 \tau^3 L_{M} \big( \| \partial_\tau p^1- \dot p(0)\|_{M_{22}}^2 + \| \dot p(0)\|_{M_{22}}^2\big).
	\end{align*}
	This finishes the proof.
\end{proof}

Before we prove the convergence, we state a discrete version of Gronwall's lemma. 
\begin{lemma}\label{lem_Gronwall}
	Suppose that $x^2_n \leq a + \sum_{i=1}^n (b_i x_i + c x_i^2)$ holds for non-negative $a$, $\{b_n\}_{n\in \N}$, $\{x_n\}_{n\in \N}$, and $c\in [0,1)$. 
	Then we have 
	\begin{equation*}
	x_n^2 \leq \Big( \sqrt{a} + \sum_{k=1}^n (1-c)^{\frac{k-2} 2} b_k \Big)^2 (1-c)^{-n} \leq \Big( \sqrt{a} + \tfrac 1 {1-c} \sum_{k=1}^n e^{\frac{-kc} 2} b_k \Big)^2 e^{nc}.
	\end{equation*}
\end{lemma}
\begin{proof}
	Define~$d_k \coloneqq  (1-c)^k$ and $\psi_k \coloneqq a + \sum_{i=1}^k (b_i x_i + c x_i^2)$ for~$k=0,1,\ldots$ We observe 
	\begin{align*}
	\psi_{n}d_{n} - a = \sum_{k=1}^n \psi_{k}d_{k} - \psi_{k-1}d_{k-1} 
	= \sum_{k=1}^n [b_{k}x_{k} +c (x_{k}^2 - \psi_{k})]d_{k-1}
	\leq \sum_{k=1}^n b_{k} \frac{\sqrt{d_{k}}}{d_1} \sqrt{\psi_{k}d_{k}}.
	\end{align*}
	With this, the first bound follows by~\cite[Lem.~8.13]{Zim21} applied to $\sqrt{\psi_{n}d_{n}}$. The second inequality is a simple implication from the first.
\end{proof}

Based on the previous three lemmata, we show that the convergence orders of the linear case still hold for state-dependent right-hand sides.
\begin{theorem}[Convergence order for state-dependent right-hand sides]\label{th_error_Lie_impl_Euler_nonlin}
	Suppose that the assumptions of Theorem~\ref{th_error_Lie_impl_Euler} are satisfied, in particular let $\tau$, $h$, and $\tau h^{-1}$ be sufficiently small. Assume that the right-hand sides $f_\Omega$ and $f_\Gamma$ are Lipschitz continuous.
	Then convergence with the same rates as in Theorem~\ref{th_error_Lie_impl_Euler} hold.
\end{theorem}
\begin{proof}
	As in the proof of Theorem~\ref{th_error_Lie_impl_Euler} the term $\err{u}^{n+1}$ denotes the difference of the solution $u(t^{n+1})$ and its approximation $u^{n+1}$, analogously $\err{p}^{n+1}$. We follow the steps of Theorem~\ref{th_error_Lie_impl_Euler}. Using the Lipschitz continuity, there exists a constant $\delta\geq 0$ depending only on $L_{A}$, $L_{M}$, $L_{A_\Gamma}$, and $L_{M_\Gamma}$, such that
	\begin{align*}
	&\|\err{u}^{n+1}\|_{M}^2  + \|\err{p}^{n+1}\|_{M_\Gamma}^2 + \sum_{k=0}^n \|\err{p}^{k+1} - \err{p}^{k}\|_{M_\Gamma}^2 + \tau \sum_{k=0}^n \|\err{u}^{k+1}\|_{A}^2 + 2 \tau \sum_{k=0}^n \|\err{p}^{k+1}\|_{A_\Gamma}^2\\
	&\quad - 2\sum_{k=0}^n \|\err{u}^{k+1}\|_{M} \big( \|u(t^{k+1})-u(t^k) - \tau \dot u (t^{k+1})\|_{M} + \tau \|\partial_\tau p^{k+1}- \partial_\tau p^k\|_{M_{22}}\big)\\
	& \quad - 2\sum_{k=0}^n \|\err{p}^{k+1}\|_{M_\Gamma} \|p(t^{k+1})-p(t^k) - \tau \dot p (t^{k+1})\|_{M_{\Gamma}} - \tau \sum_{k=0}^n \| p(t^{k+1})-p(t^k)\|_{A_{22}}^2\\
	\leq\, & 2  \tau \sum_{k=0}^n \|\err{u}^{k+1}\|_M \bigg( L_{A}^{\sfrac 1 2} \bigg\| \begin{bmatrix}
	\err{u_1}^{k+1}\\
	p(t^{k+1}) - p^k 
	\end{bmatrix} \bigg\|_A + L_{M}^{\sfrac 1 2} \bigg\| \begin{bmatrix}
	\err{u_1}^{k+1}\\
	p(t^{k+1}) - p^k 
	\end{bmatrix} \bigg\|_M \bigg)\\
	& \qquad  + \|\err{p}^{k+1}\|_{M_\Gamma} \big(L_{A_\Gamma}^{\sfrac 1 2}\| \err{p}^{k+1} \|_{A_\Gamma} + L_{M_\Gamma}^{\sfrac 1 2}\| \err{u}^{k+1} \|_{M_\Gamma}\big)\\
	\leq\, & 2  \tau \sum_{k=0}^n \|\err{u}^{k+1}\|_M \Big( L_{A}^{\sfrac 1 2} \|\err{u}^{k+1}\|_A  + L_{A}^{\sfrac 1 2}\| \err{p}^{k+1} - \err{p}^{k}\|_{A_{22}} +  L_{A}^{\sfrac 1 2}\| p(t^{k+1}) - p(t^{k})\|_{A_{22}}\\
	& \qquad  + L_{M}^{\sfrac 1 2} \|\err{u}^{k+1}\|_M + L_{M}^{\sfrac 1 2}\| \err{p}^{k+1} - \err{p}^{k}\|_{M_{22}} +  L_{M}^{\sfrac 1 2}\| p(t^{k+1}) - p(t^{k})\|_{M_{22}} \Big) \\
	& \qquad +  \|\err{p}^{k+1}\|_{M_\Gamma} \big(L_{A_\Gamma}^{\sfrac 1 2}\| \err{p}^{k+1} \|_{A_\Gamma} + L_{M_\Gamma}^{\sfrac 1 2}\| \err{u}^{k+1} \|_{M_\Gamma} \big) \\
	\leq\, & \tau \delta \sum_{k=0}^n \big(\|\err{u}^{k+1}\|_M^2 + \|\err{p}^{k+1}\|_{M_\Gamma}^2 \big) + \tau \sum_{k=0}^{n} \big( \| p(t^{k+1}) - p(t^{k})\|_{A_{22}}^2 + \| p(t^{k+1}) - p(t^{k})\|_{M_{22}}^2\big)\\
	& \quad + \big((c_M + c_{\alpha,M}) \tau h + c_A \tau h^{-1}\big)\sum_{k=0}^n \|\err{p}^{k+1} - \err{p}^{k}\|_{M_\Gamma}^2 + \tfrac 1 2 \tau \sum_{k=0}^n \|\err{u}^{k+1}\|_{A}^2 + \tau \sum_{k=0}^n \|\err{p}^{k+1}\|_{A_\Gamma}^2,
	\end{align*}
	see also~\eqref{eqn_est_Lie_impl_Euler_help_0}. By Lemma~\ref{lem_Gronwall}, we derive
	\begin{align*}
	&\ \|\err{u}^{n+1}\|_{M}^2  + \|\err{p}^{n+1}\|_{M_\Gamma}^2 + \tfrac 1 2 \tau \sum_{k=0}^n \|\err{u}^{k+1}\|_{A}^2 + \tau \sum_{k=0}^n \|\err{p}^{k+1}\|_{A_\Gamma}^2\\
	\leq &\ c e^{t^{n+1} \delta} \Big(\tau^2 \| \ddot{u}\|_{L^1(M)}^2 + \tau^2\| \ddot{p}\|_{L^1(M_{\Gamma})}^2 + \tau^2\| \dot{p}\|_{L^2(2A_{22} + M_{22})}^2 + \tau\, t^{n+1}\sum_{k=0}^n \|\partial_\tau p^{k+1}-\partial_\tau p^k\|_{M_{22}}^2 \Big)
	\end{align*}
	%
	similarly as we obtained~\eqref{eqn_est_Lie_impl_Euler_help_1}. Similar to the lines of Theorem~\ref{th_error_Lie_impl_Euler}, the last term $\tau \sum_{k=0}^n  \|\partial_\tau p^{k+1}-\partial_\tau p^k\|_{M_{22}}^2$ can be estimated by combining Lemma~\ref{lem_stability_Lie_impl_Euler_nonlin} and~\ref{lem_local_error_Lie_impl_Euler_nonlin}.
	
	The second estimate can be proven as Theorem~\ref{th_error_Lie_impl_Euler}.\ref{th_error_Lie_impl_Euler_b} with similar adjustments as in the first part of this proof.
\end{proof}
\end{document}